\documentclass[english]{amsart}

\usepackage{amssymb,amsmath,amsthm,esint,braket,mathtools,enumitem,hyperref,comment,xcolor}

\usepackage{doi}

\usepackage[margin=30mm]{geometry}

\usepackage[capitalise]{cleveref}

\usepackage[T1]{fontenc}

\usepackage{microtype}

\providecommand{\abs}[1]{\left\vert #1 \right\vert}
\providecommand{\norm}[1]{\left\Vert #1 \right\Vert}

\providecommand{\pt}[1]{\left( #1 \right)}
\providecommand{\spt}[1]{\left[ #1 \right]}

\newcommand{\domega}{{\partial\Omega}}

\newcommand{\RR}{\mathbb R}
\newcommand{\NN}{\mathbb N}

\providecommand{\A}{\mathcal A}
\providecommand{\B}{\mathcal B}
\newcommand{\D}{\,\mathrm d}

\newcommand{\dx}{{\D x}}
\newcommand{\di}{\mathrm{div}}
\newcommand{\ve}{\varepsilon}
\newcommand{\supp}{\mathrm{supp}\,}

\newcommand{\sublim}{\operatornamewithlimits{\longrightarrow}}

\newtheorem{theorem}{Theorem}[section]
\newtheorem{proposition}{Proposition}[section]
\newtheorem{lemma}{Lemma}[section]

\newtheorem{remark}{Remark}[section]

\Crefname{corollary}{Corollary}{Corollaries}
\Crefname{lemma}{Lemma}{Lemmas}
\Crefname{theorem}{Theorem}{Theorems}
\Crefname{proposition}{Proposition}{Propositions}

\begin{document}

\title{Interior regularity of some weighted quasi-linear equations}
\author{Hern{\'a}n Castro}
\email{hcastro@utalca.cl}
\address{Instituto de Matem{\'a}ticas, Universidad de Talca, Casilla 747, Talca, Chile}
\date{January 23, 2025}


\begin{abstract}
In this article we study the quasi-linear equation
\[
\left\{
\begin{aligned}
\di \A(x,u,\nabla u)&=\B(x,u,\nabla u)&&\text{in }\Omega,\\
u\in H^{1,p}_{loc}&(\Omega;w\dx)
\end{aligned}
\right.
\]
where \(\A\) and \(\B\) are functions satisfying \(\A(x,u,\nabla u)\sim \B(x,u,\nabla u)\sim w(\abs{\nabla u}^{p-2}\nabla u+\abs{u}^{p-2}u)\) for \(p>1\) and a \(p\)-admissible weight function \(w\). We establish interior regularity results of weak solutions and use those results to obtain point-wise asymptotic estimates for solutions to
\[
\left\{
\begin{aligned}
-\di(w\abs{\nabla u}^{p-2}\nabla u)&=w\abs{u}^{q-2}u&&\text{in }\Omega,\\
u\in D^{1,p,w}&(\Omega)
\end{aligned}
\right.
\]
for a critical exponent \(q>p\) in the sense of Sobolev.
\end{abstract}

\maketitle

\section{Introduction}

We are interested in obtaining some qualitative and quantitative properties of weak solutions to the following equation
\begin{equation}\label{bb-eq}
\left\{
\begin{aligned}
-\di \pt{w\abs{\nabla u}^{p-2}\nabla u}&=w\abs{u}^{q-2}u&&\text{in }\Omega\\
u&\in D^{1,p,w}(\Omega),
\end{aligned}
\right.
\end{equation}
for \(q>p>1\) critical for a weighted Sobolev embedding and \(D^{1,p,w}(\Omega)\) a weighted Sobolev space that will be made precise later. The main motivation behind studying this problem comes from the results in \cite{Cas2021} where the existence and non-existence to extremals to a Sobolev inequality with monomial weights was analyzed (see \cite{CR-O2013-2,Cas2016-2}). It is known that extremals to a Sobolev inequality can be viewed as positive solutions to \eqref{bb-eq} for an appropriate weight function \(w\), and our goal is to obtain as much information as possible regarding said extremals and, in general, of solutions to \eqref{bb-eq}.

As we mentioned above, \(w\) will be a weight function throughout this work, meaning a locally Lebesgue integrable non-negative function over an open set \(\Omega\subseteq\RR^N\) satisfying at least the following two conditions: if we abuse the notation and write \(w(B)=\int_B w\dx\) we require that \(w\) satisfies the \emph{doubling property} in \(\Omega\), meaning that there exists a \emph{doubling constant} \(\gamma>0\) such that
\begin{equation}\label{doubling-w}
w(2B)\leq \gamma w(B)
\end{equation}
holds for every (open) ball such that \(2B\subset \Omega\), where \(\rho B\) denotes the ball with the same center as \(B\) but with its radius multiplied by \(\rho>0\). The smallest possible \(\gamma>0\) for which \eqref{doubling-w} holds for every ball will be denoted by \(\gamma_w>0\) from now on. Additionally we will suppose that
\begin{equation}\label{inv-loc-int2}
0<w<\infty\qquad\lambda-\text{almost everywhere}
\end{equation}
where \(\lambda\) denotes the \(N\)-dimensional Lebesgue measure. These two conditions ensure that the measure \(w\dx\) and the Lebesgue measure \(\lambda\) are absolutely continuous with respect to each other.

Equations like \eqref{bb-eq} have been studied in the past. The most studied case is without a doubt the unweighted linear/semi-linear case, meaning \(w=1\) and \(p=2\), but also significant progress has been done for the case \(w=1\) and \(p\neq 2\). The literature is vast in both cases and we do not intend to cover everything that has been done (the interested reader could check \cite{Lions82,dFLN82,GS1981,Serrin1964,CPY2012,DaSc2006,Lin2006,Sci2016,Vet2016,GT01}, a list which is nowhere near exhaustive). Among the previously mentioned results we would like to single out a few that are relevant for this work. Firstly, Serrin \cite{Serrin1964} studied, among other things, the local regularity of solutions to the unweighted quasi-linear equation
\[
\di\A(x,u,\nabla u)=\B(x,u,\nabla u)\qquad\text{in }\Omega\subseteq\RR^N,
\]
where \(\A:\Omega\times\RR\times\RR^N\to\RR^N\) and \(\B:\Omega\times\RR\times\RR^N\to\RR\) are functions satisfying
\begin{gather*}
\label{hyp-2-serrin}\abs{\A(x,u,z)}\leq a\abs{z}^{p-1}+b\abs{u}^{p-1}+e,\\
\label{hyp-3-serrin}\abs{\B(x,u,z)}\leq c\abs{z}^{p-1}+d\abs{u}^{p-1}+f,\\
\label{hyp-1-serrin}\A(x,u,z)\cdot z\geq a^{-1}\abs{z}^p-d\abs{u}^p-g,
\end{gather*}
for a constant \(a>0\) and measurable functions \(b,c,d,e,f,g:\Omega\to [0,\infty)\) satisfying suitable integrability conditions.

Secondly, Cao, Peng and Yan \cite{CPY2012} studied for \(1<p<N\) the equation
\[
\left\{
\begin{aligned}
-\Delta_p u&=\abs{u}^{p^*-2}u+\mu\abs{u}^{p-2}u&&\text{in }\Omega\\
u&=0&&\text{on }\domega,
\end{aligned}
\right.
\]
where \(\Delta_p\) is the \(p\)-Laplace operator and \(p^*=\frac{Np}{N-p}\) is the critical Sobolev exponent. The main result of that work deals with the existence of infinitely many solutions to said equation, however we would like to point out a result from the Appendix of that work. There it is shown that weak solutions to 
\begin{equation}\label{CPY-eq}
\left\{
\begin{aligned}
-\Delta_p u&=\abs{u}^{p^*-2}u&&\text{in }\RR^N\\
u&\in W^{1,p}(\RR^N)
\end{aligned}
\right.
\end{equation}
satisfy the following decay estimate
\[
\abs{u(x)}\leq C\pt{\frac{1}{1+\abs{x}}}^{\frac{N-p}{p-1}-\theta}
\]
for any \(\theta>0\). This result almost captures the behavior satisfied by positive solutions (e.g. extremals \(U(x)\) to the Sobolev inequality): recently it has been shown by \cite{Vet2016,Sci2016} that any positive solution to \eqref{CPY-eq} must be of the form
\[
C_{N,p}\pt{\frac{a^{\frac{1}{p-1}}}{a^\frac{p}{p-1}+\abs{x-x_0}^{\frac{p}{p-1}}}}^{\frac{N-p}{p}}\sim  C\pt{\frac{1}{1+\abs{x}}}^{\frac{N-p}{p-1}}\qquad\text{for large }\abs{x}, 
\]
for some \(a>0\) and \(x_0\in\RR^N\).

The weighted case \(w\neq 1\) has also been studied and there has been an important progress in this situation. As before, we do not plan to give an exhaustive list of such works but the interested reader might want to look at the following list of papers \cite{Muck1972,MuckWhee1974,MS68,Tru1971,FaKeSe1982,Fra1991,Zam2002,Gut1989,Fer2006,ChaWhe1986,MRW2012,MRW2015,Bon2018,Garain2020}. As with the unweighted case we would like to single out a few of the works dealing with weights. One line of progress that we would like to follow regarding weighted quasilinear equations perhaps begins with the works of Muckenhoupt and Wheeden \cite{Muck1972,MuckWhee1974} who characterized the weights \(w\) for which one has the boundedness in the \(L^p(\RR^N,w\dx)\) space (\(1<p<\infty\)) of two important operators: the Hardy's maximal function defined as 
\[
Mf(x)=\sup_{B}\lambda(B)^{-1}\int_{B} f\dx,
\]
where the supremum is taken over all balls containing \(x\), and the singular integral operator
\[
T(f)(x)=\int_{\RR^N}f(y)\abs{x-y}^{1-N}\dx
\]
In \cite{Muck1972} it was shown that Hardy's maximal operator \(M\) is bounded in the weighted \(L^p(\Omega,w\dx)\) if and only if the weight \(w\) belongs to the class \(A_p\), that is if there exists a constant \(C>0\) such that
\begin{equation}\label{Ap-weight}
\pt{\int_{B} w\dx}\pt{\int_{B} w^{-\frac{1}{p-1}}\dx}^{p-1}\leq C\lambda(B)^p
\end{equation}
holds for every ball \(B\). And in \cite{MuckWhee1974} it was proved that the singular integral operator \(T\) is bounded in the weighted \(L^p\) space if \(w\) is in \(A_q\) for some \(q>1\).

From the results of \cite{Muck1972,MuckWhee1974} we can then look at the work of Fabes, Kenig and Serapioni \cite{FaKeSe1982} who obtained local regularity results (for instance a Harnack inequality and the interior Hölder regularity of weak solutions) for the operator \(\di (A(x)\nabla u)\) where \(A(x)\) is a matrix valued function with eigenvalues behaving like \(w(x)\) for some \(w\in A_2\). The main tool used in \cite{FaKeSe1982} is De Giorgi variable truncation method \cite{DeGio1957}, and a rather important insight from \cite{FaKeSe1982} is that for De Giorgi's method to work the key ingredient is having a local Sobolev inequality: there must be constants \(k>1\) and \(C>0\) such that
\begin{equation*}
\pt{\fint_B \abs{u}^{2k}w\dx}^{\frac1{2k}}\leq Cr\pt{\fint_B\abs{\nabla u}^2w\dx}^{\frac12}\qquad \forall u\in C^\infty_c(B),
\end{equation*}
where \(B\) is a ball of radius \(r>0\) and \(\fint_Bfw=w(B)^{-1}\int_Bfw\), inequality that holds true for \(w\in A_2\) and suitable \(k>1\) depending on \(w\).

In order to present the main results of this work we need to make a brief detour and say a few words regarding \emph{weighted Sobolev spaces}. The spaces we will use are defined as subspaces of the following weighted Lebesgue space
\[
L^{p,w}(\Omega)=\set{u:\Omega\to\RR \text{ measurable}: \int_{\Omega}\abs{u}^pw\dx<\infty}
\]
equipped with the norm
\[
\norm{u}_{p,w}^p=\int_{\Omega}\abs{u}^pw\dx.
\]

In order to work with weighted Sobolev spaces one needs to establish some conditions over the weights so that the corresponding spaces have sufficient structural properties. In addition to the basic conditions \eqref{doubling-w} and \eqref{inv-loc-int2}, we will suppose that the weight \(w\) satisfies a Poincaré inequality: it is known (see for instance \cite[Chapter 20]{HeKiMa2006}) that if \(w\) is a doubling weight satisfying 
\begin{enumerate}[label={(P{\scriptsize\Roman*})}]
\item\label{weighted-poincare}  \emph{Local weighted \((1,p)\)-Poincaré inequality}: There exists \(\rho\geq 1\) such that if \(u\in C^1(\Omega)\) then for all balls \(B\subset\Omega\) of radius \(l(B)\) one has
\[
\frac{1}{w(B)}\int_B\abs{u-u_{B,w}}w\dx\leq C_1l(B)\pt{\frac{1}{w(\rho B)}\int_{\rho B}\abs{\nabla u}^pw\dx}^{\frac1p}
\]
where \(w(B)=\int_B w\dx\) and 
\begin{equation*}
u_{B,w}=\frac{1}{w(B)}\int_B uw\D x
\end{equation*}
is the weighted average of \(u\) over \(B\),
\end{enumerate}
then \(w\) is automatically \emph{\(p\)-admissible}, that is, it also satisfies the following properties
\begin{enumerate}[resume*]
\item\label{uniq-grad} \emph{Uniqueness of the gradient}: If \((u_n)_{n\in\NN}\subseteq C^1(\Omega)\) satisfy 
\[
\int_\Omega\abs{u_n}^pw\dx\sublim\limits_{n\to\infty} 0\quad\text{and}\quad\int_{\Omega}\abs{\nabla u_n-v}^pw\dx\sublim\limits_{n\to\infty} 0
\]
for some \(v:\Omega\to \RR^N\), then \(v=0\),

\item \emph{Local Poincaré-Sobolev inequality}: There exist constants \(C_3>0\) and \(\chi>1\) such that for all balls \(B\subset\Omega\) one has
\[
\pt{\frac{1}{w(B)}\int_B\abs{u-u_{B,w}}^{\chi p}w\dx}^{\frac1{\chi p}}\leq C_2l(B)\pt{\frac1{w(B)}\int_B \abs{\nabla u}^pw\dx}^{\frac 1p}
\]
for bounded \(u\in C^1(B)\),

\item \label{sob-ineq}
\emph{Local Sobolev inequality}: There exist constants \(C_2>0\) and \(\chi>1\) (same as above) such that for all balls \(B\subset\Omega\) one has
\[
\pt{\frac{1}{w(B)}\int_B\abs{u}^{\chi p}w\dx}^{\frac1{\chi p}}\leq C_2l(B)\pt{\frac1{w(B)}\int_B \abs{\nabla u}^pw\dx}^{\frac 1p}
\]
for \(u\in C^1_c(B)\).

\end{enumerate}

\begin{remark}\label{local-sobolev-remark}
Suppose for a moment that the weight \(w\) verifies the following: there exist constants \(C,D>0\) such that 
\begin{equation}\label{w-ball-estimate}
\frac{w(B_R(y))}{w(B_r(x))}\leq C\pt{\frac{R}r}^{D},\qquad \text{for all }0<r\leq R<\infty \text{ with }B_r(x)\subseteq B_R(y)\subseteq\Omega.
\end{equation}
It follows from \cite[Theorem 5.1 and Corollary 9.8]{HaKo2000} and \cite[Theorem 4]{AlGoHa2020} that
\begin{equation}\label{local-Poincare}
\pt{\frac{1}{w(B)}\int_B\abs{u-u_{B,w}}^{\frac{Dp}{D-p}}w\dx}^{\frac{D-p}{Dp}}\leq C_2l(B)\pt{\frac1{w(B)}\int_B \abs{\nabla u}^pw\dx}^{\frac 1p}\quad\forall\, u \in C^\infty(B)
\end{equation}
and as a consequence we also have
\begin{equation}\label{local-Sobolev}
\pt{\frac{1}{w(B)}\int_B\abs{u}^{\frac{Dp}{D-p}}w\dx}^{\frac{D-p}{Dp}}\leq C_2l(B)\pt{\frac1{w(B)}\int_B \abs{\nabla u}^pw\dx}^{\frac 1p}\quad\forall\, u \in C^\infty_c(B).
\end{equation}

On the one hand, observe that this is precisely what happens for the \(N\)-dimensional Lebesgue measure \(\lambda\) as one can see that it satisfies
\[
\frac{\lambda(B_R)}{\lambda(B_r)}=\pt{\frac{R}{r}}^N,
\]
so in the unweighted case one has \(D=N\) and the classical local Poincaré-Sobolev and local Sobolev inequalities are readily recovered.

On the other hand, \eqref{w-ball-estimate} is automatically satisfied for any doubling weight \(w\) because one can iterate \eqref{doubling-w} to obtain
\begin{equation*}
\frac{w(B_R)}{w(B_r)}\leq C\pt{\frac{R}r}^{\log_2 \gamma_w}.
\end{equation*}

With the above observation in mind it is appropriate to denote by \(D_w=\log_2\gamma_w\) and to call \(D_w\) the \emph{dimension of the weight} \(w\) (this is related to the Assouad dimension of the measure \(w\dx\), see \cite{Assouad1979,FraHow2020,Hei2001}). Also we will use the notation \(\chi_w=\frac{D_w}{D_w-p}\).
\end{remark}

As mentioned at the beginning, the above properties are useful in the definition of weighted Sobolev spaces: for an open set \(\Omega\subseteq\RR^N\) we define the weighted Sobolev space \(H^{1,p,w}(\Omega)\)
\begin{equation}\label{def-H}
H^{1,p,w}(\Omega)=\text{the completion of }\set{u\in C^1(\Omega): u,\frac{\partial u}{\partial x_i}\in L^{p,w}(\Omega) \text{ for all }i}
\end{equation}
equipped with the norm
\begin{equation}\label{sobolev-norm}
\norm{u}_{1,p,w}^p=\norm{u}_{p,w}^p+\sum_{i=1}^N\norm{\frac{\partial u}{\partial x_i}}_{p,w}^p.
\end{equation}
Observe that property \ref{uniq-grad} guarantees that functions in this space have a unique gradient (see \cite[Section 1.9]{HeKiMa2006}).

Having established the ambient space we can now give the main results of this article. As we mentioned before the goal of this work is to obtain qualitative and quantitative properties of weak solutions to \eqref{bb-eq}. To do so we first study the local regularity of weak solutions the following quasi-linear problem
\begin{equation}\label{model-eq}
\left\{
\begin{aligned}
\di \A(x,u,\nabla u)&=\B(x,u,\nabla u),&&\text{in }\Omega\subseteq\RR^N\\
u&\in H^{1,p,w}_{loc}(\Omega),
\end{aligned}
\right.
\end{equation}
where \(\A:\Omega\times\RR\times\RR^N\to\RR^N\) and \(\B:\Omega\times\RR\times\RR^N\to \RR\) are functions verifying the Serrin-like conditions
\begin{gather}
\label{hyp-A1}\tag{H1}
\A(x,u,z)\cdot z\geq w(x)\pt{\abs{z}^p-d\abs{u}^p-g},\\
\label{hyp-A2}\tag{H2}
\abs{\A(x,u,z)}\leq w(x)\pt{\abs{z}^{p-1}+b\abs{u}^{p-1}+e},\\
\label{hyp-B}\tag{H3}
\abs{\B(x,u,z)}\leq w(x)\pt{c\abs{z}^{p-1}+d\abs{u}^{p-1}+f},
\end{gather}
for measurable functions \(b,c,d,e,f,g:\Omega\to\RR^+\) satisfying suitable integrability conditions that will vary from theorem to theorem. With the above into consideration, throughout the rest of this article the function \(w\) will be a non-negative locally integrable weight functions satisfying \eqref{doubling-w}, \eqref{inv-loc-int2} and the local weighted \((1,p)\)-Poincaré inequality \ref{weighted-poincare}.

The first result deals with the local boundedness of weak solutions, namely we have

\begin{theorem}\label{thm-local-bdd}
Suppose that \(1<p<D_w\) where \(D_w\) is the dimension of the weight \(w\) defined at \cref{local-sobolev-remark}, additionally suppose that for \(0<\ve<1\) one has 
\begin{equation}\label{struc-hyp}
b,e\in L^{\frac{D_w}{p-1},w}(\Omega),\quad
c\in L^{\frac{D_w}{1-\ve},w}(\Omega),\quad\text{and}\quad
d,f,g\in L^{\frac{D_w}{p-\ve},w}(\Omega).
\end{equation}
For fixed \(x_0\in \Omega\) and \(R>0\) such that \(B_{2R}(x_0)\subset\Omega\) suppose \(u\in H^{1,p,w}_{loc}\) is a local weak solution to
\[
\di\A(x,u,\nabla u)=\B(x,u,\nabla u) \qquad\text{in }B_{2R}(x_0)
\]
then
\[
\norm{u}_{L^{\infty}(B_R(x_0))}\leq C_R\spt{\pt{\fint_{B_{2R}(x_0)}\abs{u}^pw\dx}^{\frac1p}+k_R},
\]
where \(C_R>0\) depends on \(\ve,\ D_w,\ N, a, p\) and the quantities 
\begin{gather*}
b_R:=R^{p-1}\pt{\fint_{B_{2R(x_0)}}\abs{b}^{\frac{D_w}{p-1}}w}^{\frac{p-1}{D_w}},\quad c_R:=R\pt{\fint_{B_{2R(x_0)}}\abs{c}^{\frac{D_w}{1-\ve}}w}^{\frac{1-\ve}{D_w}}, \\ d_R:=R^{p}\pt{\fint_{B_{2R(x_0)}}\abs{d}^{\frac{D_w}{p-\ve}}w}^{\frac{p-\ve}{D_w}},\quad 
e_R:=R^{p-1}\pt{\fint_{B_{2R(x_0)}}\abs{e}^{\frac{D_w}{p-1}}w}^{\frac{p-1}{D_w}},\\
f_R:=R^p\pt{\fint_{B_{2R(x_0)}}\abs{f}^{\frac{D_w}{p-\ve}}w}^{\frac{p-\ve}{D_w}},\quad
g_R:=R^p\pt{\fint_{B_{2R(x_0)}}\abs{g}^{\frac{D_w}{p-\ve}}w}^{\frac{p-\ve}{D_w}}.
\end{gather*}
The constant \(k_R\) is defined by
\[
k_R=\pt{e_R+f_R}^{\frac{1}{p-1}}+g_R^{\frac1p}.
\]
\end{theorem}

The above result requires that \(\ve>0\) in \eqref{struc-hyp}. If we consider the case  \(\ve=0\) we no longer have local boundedness, but we do obtain that \(u\) is \(L^{s,w}\) integrable for all \(s<\infty\) as the following results shows

\begin{theorem}\label{thm-local-inte}
Let \(1<p<D_w\) where \(D_w\) is the dimension of the weight \(w\) defined at \cref{local-sobolev-remark}. Suppose that one has 
\begin{equation}\label{struc-hyp2}
b,e\in L^{\frac{D_w}{p-1},w}(\Omega),\quad
c\in L^{D_w,w}(\Omega),\quad\text{and}\quad
d,f,g\in L^{\frac{D_w}{p},w}(\Omega).
\end{equation}
For fixed \(x_0\in \Omega\) and \(R>0\) such that \(B_{2R}(x_0)\subset\Omega\) suppose \(u\in H^{1,p,w}_{loc}\) is a local weak solution to
\[
\di\A(x,u,\nabla u)=\B(x,u,\nabla u) \qquad\text{in }B_{2R}(x_0)
\]
then for all \(1\leq s<\infty\) we have
\[
\pt{\fint_{B_{R}(x_0)}\abs{u}^{s}w\dx}^{\frac1s}\leq C_{R,s}\spt{\pt{\fint_{B_{2R}(x_0)}\abs{u}^pw\dx}^{\frac1p}+k_{R}},
\]
for some constant \(C_{R,s}>0\) depending on \(s,\ D_w,\ N\) and the structure of \(\A\) and \(\B\), namely \(p,\ a>0\) and the quantities 
\begin{gather*}
b_R:=R^{p-1}\pt{\fint_{B_{2R(x_0)}}\abs{b}^{\frac{D_w}{p-1}}w}^{\frac{p-1}{D_w}},\quad c_R:=R\pt{\fint_{B_{2R(x_0)}}\abs{c}^{D_w}w}^{\frac{1}{D_w}},\\
d_R:=R^{p}\pt{\fint_{B_{2R(x_0)}}\abs{d}^{\frac{D_w}{p}}w}^{\frac{p}{D_w}},\quad
e_R:=R^{p-1}\pt{\fint_{B_{2R(x_0)}}\abs{e}^{\frac{D_w}{p-1}}w}^{\frac{p-1}{D_w}},\\
f_R:=R^p\pt{\fint_{B_{2R(x_0)}}\abs{f}^{\frac{D_w}{p}}w}^{\frac{p}{D_w}},\quad
g_R:=R^p\pt{\fint_{B_{2R(x_0)}}\abs{g}^{\frac{D_w}{p}}w}^{\frac{p}{D_w}}.
\end{gather*}
The value of \(k_R\) is 
\[
k_R=\pt{e_R+f_R}^{\frac{1}{p-1}}+g_R^{\frac1p}.
\]
\end{theorem}

If we keep relaxing the integrability conditions on the structural parameters \(b,c,\ldots, g\) we still are able to obtain some integrability of \(u\) as the following result shows

\begin{theorem}\label{thm-local-lower}
Let \(1<p<D_w\) where \(D_w\) is the dimension of the weight \(w\) defined at \cref{local-sobolev-remark}. Suppose that one has 
\begin{equation}\label{struc-hyp1'}
\begin{aligned}
b\in L^{\frac{D_w}{p-1},w}(\Omega),\quad
c\in L^{D_w,w}(\Omega),\quad
d\in L^{\frac{D_w}{p},w}(\Omega),\\
e\in L^{\frac{D_wr}{D_w-r},w}(\Omega),\quad
f\in L^{r,w}(\Omega),\quad\text{and}\quad
g\in L^{t,w}(\Omega),
\end{aligned}
\end{equation}
where \(r,t\) verify
\[
\frac{1}{p-1}\pt{\frac{1}{r}-\frac{p}{D_w}}=\frac{1}{p}\pt{\frac{1}{t}-\frac{p}{D_w}}=\frac1{s}
\]
for some \(s\geq \frac{D_wp}{D_w-p}\). For fixed \(x_0\in \Omega\) and \(R>0\) such that \(B_{2R}(x_0)\subset\Omega\) suppose \(u\in H^{1,p,w}_{loc}\) is a local weak solution to
\[
\di\A(x,u,\nabla u)=\B(x,u,\nabla u) \qquad\text{in }B_{2R}(x_0)
\]
then we have
\[
\pt{\fint_{B_{R}(x_0)}\abs{u}^{s}w\dx}^{\frac1s}\leq C_{R,s}\spt{\pt{\fint_{B_{2R}(x_0)}\abs{u}^pw\dx}^{\frac1p}+k_{R}},
\]
for some constant \(C_{R,s}>0\) depending on \(s,\ D_w,\ N\) and the structure of \(\A\) and \(\B\), namely \(p,\ a\) and the quantities 
\begin{gather*}
b_R:=R^{p-1}\pt{\fint_{B_{2R(x_0)}}\abs{b}^{\frac{D_w}{p-1}}w}^{\frac{p-1}{D_w}},\quad c_R:=R\pt{\fint_{B_{2R(x_0)}}\abs{c}^{D_w}w}^{\frac{1}{D_w}}, \\ d_R:=R^{p}\pt{\fint_{B_{2R(x_0)}}\abs{d}^{\frac{D_w}{p}}w}^{\frac{p+\ve}{D_w}}. 
\end{gather*}
The value of \(k_R\) is 
\[
k_R=\pt{e_R+f_R}^{\frac{1}{p-1}}+g_R^{\frac1p},
\]
where
\begin{gather*}
e_R:=R^{p-1}\pt{\fint_{B_{2R(x_0)}}\abs{e}^{\frac{D_wr}{D_w-r}}w}^{\frac{1}{\chi_w r}},\quad
f_R:=R^{p}\pt{\fint_{B_{2R(x_0)}}\abs{f}^{r}w}^{\frac{1}{r}},\\
g_R:=R^{p}\pt{\fint_{B_{2R(x_0)}}\abs{g}^{t}w}^{\frac{1}{t}}.
\end{gather*}
\end{theorem}

Next we show the validity of the Harnack inequality, that is we have

\begin{theorem}[Harnack inequality]\label{harnack-thm}
Under the same hypotheses of \cref{thm-local-bdd} with the additional assumption that \(u\) is a non-negative weak solution of \(\di\A=\B\) in \(B_{3R}\) then
\begin{equation}
\max_{B_{R}}u\leq C_R\pt{\min_{B_R}u+k_R} 
\end{equation}
where \(C_R\) and \(k_R\) are as in \cref{thm-local-bdd}.
\end{theorem}

And as it is usual in this context, once one is able to obtain the Harnack inequality, then a local Hölder regularity result readily follows

\begin{theorem}\label{holder-thm}
Suppose that \(1<p<D_w\) and that the hypotheses of \cref{thm-local-bdd} are satisfied, but in addition we suppose that
\[
b\in L^{\frac{D_w}{p-1-\ve},w}(\Omega)
\]
for some \(\ve>0\). If \(u\in H^{1,p,w}\) is a (local) weak solution of
\[
\di\A(x,u,\nabla u)=\B(x,u,\nabla u) \text{ in }\Omega,
\]
then \(u\) is locally Hölder continuous.
\end{theorem}

With the above local regularity results we are now able to study \eqref{bb-eq}, that is we now turn to the equation
\begin{equation}\label{crit-weighted-sob-eq}
\left\{
\begin{aligned}
-\di\pt{w\abs{\nabla u}^{p-2}\nabla u}&=w\abs{u}^{q-2}u&&\text{in }\Omega,\\
u&\in D^{1,p,w}(\Omega),
\end{aligned}
\right.
\end{equation}
where \(q=\frac{D_wp}{D_w-p}\) is the \emph{local} weighted Sobolev exponent associated to the weight \(w\) and \(D^{1,p,w}(\Omega)\) is the closure of \(C^\infty_c(\Omega)\) under the (semi) norm \(\norm{\nabla u}_{p,w}\). If we further suppose that the weight \(w\) is such that the \emph{global} weighted Sobolev inequality 
\begin{equation}\label{sob-ineq-w}
S_{p,w}\pt{\int_{\Omega}\abs{u}^{q}w\dx}^{\frac{1}{q}}\leq \pt{\int_\Omega\abs{\nabla u}^pw\dx}^{\frac1p}
\end{equation}
is satisfied for all \(u\in C^\infty_c(\Omega)\) then one has \(D^{1,p,w}(\Omega)\hookrightarrow L^{q,w}(\Omega)\) and we are able to prove the following decay estimate for weak solutions to the critical equation \eqref{crit-weighted-sob-eq}.

\begin{theorem}[Decay estimate of weak solutions]\label{decay-thm}
Suppose \(w\) satisfies \eqref{sob-ineq-w} in addition to \eqref{doubling-w}, \eqref{inv-loc-int2} and \ref{weighted-poincare}. If \(u\in D^{1,p,w}(\Omega)\) is a weak solution of
\[
-\di(w\abs{\nabla u}^{p-2}\nabla u)=w\abs{u}^{q-2}u\qquad\text{in }\Omega,
\]
there exist \(R>0\), \(C>0\) and \(\lambda>0\) such that if \(x\in\Omega\) satisfies \(\abs{x}\geq R\) then
\[
\abs{u(x)}\leq \frac{C}{\abs{x}^{\frac{D_w-p}{p}+\lambda}}.
\]
\end{theorem}
\begin{remark}
If we know specifics about the weight \(w\), then \cref{decay-thm} can be improved to obtain that weak solutions satisfy
\[
\abs{u(x)}\leq \frac{C}{\abs{x}^{\frac{D_w-p}{p-1}-\ve}}
\]
for large \(\abs{x}\) and any \(\ve>0\). To obtain such result one uses \cref{decay-thm} together with a comparison principle. We will show this later for the case of monomial weights \(w(x)=x^A\) and power type weights \(w(x)=\abs{x}^a\) in \cref{examples-sect}.
\end{remark}

The rest of this article is dedicated to the proofs of the above results. In \cref{sec-apriori} we study \eqref{model-eq} and obtain the proofs of \cref{thm-local-bdd,thm-local-inte,thm-local-lower,harnack-thm,holder-thm} whereas in \cref{sec-infty} we turn to the proof of \cref{decay-thm}. Finally in \cref{examples-sect} we exhibit specific examples of weights for which our results apply and how the specificity of each weight allows us to improve on \cref{decay-thm} to obtain sharper results.

\section{Local estimates}\label{sec-apriori}

We follow the proof of analogous results given in \cite{Serrin1964}. The main ingredient will be Moser's iteration technique which relies on using a suitable version of \(u^\beta\) as a test function for \eqref{model-eq}. To do so we consider \(\alpha\geq 1\) and \(0\leq k\leq l\) and we define  \(F:[k,\infty)\to\RR\) as
\begin{equation}\label{func-F}
F(x)=F_{\alpha,k,l}(x)=\begin{dcases}
x^{\alpha}&\text{if }k\leq x\leq l,\\
l^{\alpha-1}\pt{\alpha x-(\alpha -1)l}&\text{if }x>l.
\end{dcases}
\end{equation}
Observe that \(F\in C^1([k,\infty))\) with  \(\abs{F'(x)}\leq \alpha l^{\alpha-1}\). We also consider \(\bar x=\abs{x}+k\) and \(G:\RR\to\RR\) defined as
\begin{equation}\label{func-G}
G(x)=G_{\alpha,k,l}(x)=\mathrm{sign}(x)\pt{F(\bar x)\abs{F'(\bar x)}^{p-1}-\alpha^{p-1}k^\beta}
\end{equation}
where \(\beta=1+p(\alpha-1)\). Then clearly \(G\) is a piecewise smooth function which is linear if \(\abs{x}>l-k\) and that both \(F\) and \(G\) satisfy
\begin{gather*}
\abs{G(x)}\leq F(\bar x)\abs{F'(\bar x)}^{p-1}\\
\bar xF'(\bar x)\leq \alpha F(\bar x)\\
F'(\bar x)\leq \alpha F(\bar x)^{1-\frac{1}{\alpha}}
\end{gather*}
and
\[
G'(x)=\begin{dcases}
\frac{\beta}{\alpha}\abs{F'(\bar x)}^p&\text{if }\abs{x}<l-k,\\
\abs{F'(\bar x)}^p&\text{if }\abs{x}>l-k.
\end{dcases}
\]

Observe that as \(l\to \infty\) we have \(F(\bar x)\sim \bar x^\alpha\) and \(G(x)\sim \frac{\beta}{\alpha}\bar x^{\beta}\) so the function \(G(x)\) will play the role of \(x^\beta\). Finally, observe that if \(\eta\in C^\infty_c(\Omega)\) and if \(u\in H^{1,p,w}_{loc}(\Omega)\) then \(\varphi=\eta^pG(u)\) is a valid test function in
\[
\int_\Omega \A(x,u,\nabla u)\nabla\varphi+\B(x,u,\nabla u)\varphi=0
\]
thanks to the results in \cite[Chapter 1]{HeKiMa2006} regarding weighted Sobolev spaces for \(p\)-admissible weights we can now prove the local boundedness of weak solutions.

Since no confusion is present, throughout the proofs in this chapter we will write \(D=D_w\) and \(\chi=\chi_w=\frac{D}{D-p}>1\).

\begin{proof}[Proof of \cref{thm-local-bdd}]
For the case \(R=1\) and \(x_0\in\Omega\) such that \(B_2(x_0)\subset\Omega\) we follow the proof of \cite[Theorem 1]{Serrin1964} with a few modifications: recall that for such \(\chi\) then the local Sobolev inequality \eqref{local-Sobolev} is valid, and by using \eqref{hyp-A1}-\eqref{hyp-B} we can write
\begin{equation}\label{hyp-mod}
\begin{aligned}
\abs{\A(x,u,z)}\leq w\pt{\abs{z}^{p-1}+\bar b\bar u^{p-1}},\\
\abs{\B(x,u,z)}\leq w\pt{c\abs{z}^{p-1}+\bar d\bar u^{p-1}},\\
\A(x,u,z)\cdot z\geq w\pt{\abs{z}^{p}-\bar d\bar u^{p}},
\end{aligned}
\end{equation}
where
\begin{gather*}
\bar b=b+k^{1-p}e,\\
\bar d=d+k^{1-p}f+k^{-p}g,
\end{gather*}
and \(\bar u=\abs{u}+k\) for \(k\geq 0\) defined as\footnote{If \(e=f=g=0\) we can take any \(k>0\) and at the very end we can pass to the limit \(k\to 0^+\).}  
\begin{equation*}
k=\pt{\pt{\fint_{B_2}\abs{e}^{\frac{D}{p-1}}w}^{\frac{p-1}{D}}+\pt{\fint_{B_2}\abs{f}^{\frac{D}{p-\ve}}w}^{\frac{p-\ve}{D}}}^{\frac{1}{p-1}}+\pt{\pt{\fint_{B_2}\abs{g}^{\frac{D}{p-\ve}}w}^{\frac{p-\ve}{D}}}^{\frac1p},
\end{equation*}
where to simplify the notation from now on we use \(B_2\) to denote \(B_2(x_0)\). Observe that this choice of \(k\) together with the assumptions over the functions \(b\) through \(g\) imply that
\begin{equation}\label{k-esti}
\fint_{B_2}\abs{\bar b}^{\frac{D}{p-1}}w\leq C,\qquad
\fint_{B_2}\abs{\bar d}^{\frac{D}{p-\ve}}w\leq C,
\end{equation}
for some constant \(C>0\) depending on \(b,d,e,f,g\) and \(w(B_2)\).

As we mentioned at the beginning of this section for the local weak solution \(u\) and arbitrary non-negative \(\eta\in C^\infty_c(B_2)\) we can consider \(\varphi=\eta^pG(u)\) as a valid test function and with the aid of \eqref{hyp-mod} one can obtain the a.e. estimate
\begin{align*}
\A\cdot\nabla \varphi+\B\varphi&=\eta^p G'(u)\A\cdot\nabla u+p\eta^{p-1}G(u)\A\cdot\nabla \eta+\eta^pG(u)\B\\
&\geq \eta^p G'(u)w\pt{\abs{\nabla u}^{p}-\bar d\bar u^{p}}
-p\eta^{p-1}\abs{\nabla \eta G(u)}w\pt{\abs{\nabla u}^{p-1}+\bar b\bar u^{p-1}}\\
&\quad-\eta^p\abs{G(u)}w\pt{c\abs{\nabla u}^{p-1}+\bar d\bar u^{p-1}}\\
&\geq \abs{\eta F'(\bar u)\nabla u}^pw-p\abs{F(\bar u)\nabla \eta}\abs{\eta F'(\bar u)\nabla u}^{p-1}w-p\bar b\abs{F(\bar u)\nabla \eta}\abs{\eta \bar u F'(\bar u)}^{p-1}w\\
&\quad -c\eta F(\bar u)\abs{\eta F'(\bar u)\nabla u}^{p-1}w-\bar d\pt{\alpha^{-1}\beta\abs{\eta \bar u F'(\bar u)}^p+\eta F(\bar u)\abs{\eta \bar u F'(\bar u)}^{p-1}}w
\end{align*}
so that if \(v=F(\bar u)\) one reaches
\begin{multline}\label{bas-esti0}
\A\cdot\nabla \varphi+\B\varphi\geq
\abs{\eta \nabla v}^pw-p\abs{v\nabla \eta}\abs{\eta \nabla v}^{p-1}w
-p\alpha^{p-1}\bar b\abs{v\nabla \eta}\abs{\eta v}^{p-1}w\\
-c\eta v\abs{\eta \nabla v}^{p-1}w
-\pt{1+\beta}\alpha^{p-1}\bar d\abs{\eta v}^{p}w.
\end{multline}
After averaging the above inequality over \(B_2\) we obtain
\begin{multline}\label{bas-esti}
\fint_{B_2}\abs{\eta\nabla v}^pw\leq p\fint_{B_2}\abs{v\nabla \eta}\abs{\eta\nabla v}^{p-1}w+p\alpha^{p-1}\fint_{B_2}\bar b\abs{v\nabla \eta}\abs{v\eta}^{p-1}w\\+\fint_{B_2} cv\eta\abs{\eta\nabla v}^{p-1}w+(1+\beta)\alpha^{p-1}\fint_{B_2}\bar d\abs{v\eta}^pw,
\end{multline}
and each term on the right hand side can be estimated using \eqref{local-Sobolev}, \eqref{struc-hyp}, and \eqref{k-esti} as follows: 
\begin{equation}\label{gath1}
\fint_{B_2}\abs{v\nabla \eta}\abs{\eta\nabla v}^{p-1}w\leq \pt{\fint_{B_2}\abs{v\nabla \eta}^pw}^{\frac1p}\pt{\fint_{B_2}\abs{\eta\nabla v}^{p}w}^{1-\frac1p}
\end{equation}

\begin{equation}\label{gath2}
\fint_{B_2}\bar b\abs{v\nabla \eta}\abs{v\eta}^{p-1}w
\leq C\pt{\pt{\fint_{B_2}\abs{v\nabla\eta}^{p}w}+\pt{\fint_{B_2}\abs{v\nabla \eta}^pw}^{\frac1p}\pt{\fint_{B_2}\abs{\eta \nabla v}^{p}w}^{1-\frac{1}{p}}}
\end{equation}

\begin{multline}\label{gath3}
\fint_{B_2} cv\eta\abs{\eta\nabla v}^{p-1}w
\leq C\pt{\fint_{B_2} \abs{v\eta}^p w}^{\frac{\ve}p}\left(\pt{\fint_{B_2} \abs{v\nabla \eta}^{p} w}^{\frac{1-\ve}p}\pt{\fint_{B_2} \abs{\eta\nabla v}^{p}w}^{1-\frac1p}\right.\\
\left.+\pt{\fint_{B_2} \abs{\eta\nabla v}^{p}w}^{1-\frac{\ve}p}\right)
\end{multline}

\begin{equation}\label{gath4}
\fint_{B_2}\bar d\abs{v\eta}^pw
\leq C\pt{\fint_{B_2}\abs{v\eta}^{p}w}^{\frac{\ve}{p}}\pt{\pt{\fint_{B_2}\abs{v\nabla\eta}^{p}w}^{\frac{p-\ve}{p}}+\pt{\fint_{B_2}\abs{\eta\nabla v}^{p}w}^{\frac{p-\ve}{p}}}
\end{equation}

therefore if one considers  
\[
z=\frac{\pt{\fint_{B_2}\abs{\eta\nabla v}^pw}^{\frac1p}}{\pt{\fint_{B_2}\abs{v\nabla \eta}^pw}^{\frac1p}}
\]
and
\[
\zeta=\frac{\pt{\fint_{B_2}\abs{\eta v}^pw}^{\frac1p}}{\pt{\fint_{B_2}\abs{v\nabla \eta}^pw}^{\frac1p}}
\]
then, because \(\alpha\geq 1\), \eqref{bas-esti} becomes
\[
z^p\leq C\pt{z^{p-1}+\alpha^{p-1}(1+z^{p-1})+\zeta^\ve(z^{p-1}+z^{p-\ve})+(1+\beta)\alpha^{p-1}\zeta^\ve(1+z^{p-\ve})}
\]
which with the aid of \cite[Lemma 2]{Serrin1964} gives
\[
z\leq C\alpha^{\frac{p}{\ve}}(1+\zeta).
\]
The above translates to
\begin{equation}\label{fund-esti}
\pt{\fint_{B_2}\abs{\eta \nabla v}^pw}^{\frac1p}\leq C\alpha^{\frac{p}{\ve}}\pt{\pt{\fint_{B_2}\abs{\eta v}^pw}^{\frac1p}+\pt{\fint_{B_2}\abs{v\nabla \eta}^pw}^{\frac1p}}
\end{equation}
and because of \eqref{local-Sobolev} we obtain
\begin{equation}\label{fund-esti2}
\pt{\fint_{B_2}\abs{\eta v}^{\chi p}w}^{\frac1{\chi p}}\leq C\alpha^{\frac{p}{\ve}}\pt{\pt{\fint_{B_2}\abs{\eta v}^pw}^{\frac1p}+\pt{\fint_{B_2}\abs{v\nabla \eta}^pw}^{\frac1p}}.
\end{equation}

To continue one considers a sequence of cut-off functions as follows: we take \(\eta_n\in C^\infty_c(B_{h_n})\) such that \(\eta_n\equiv 1\) in \(B_{h_{n+1}}\) and \(\abs{\nabla \eta_n}\leq C2^n\) where \(h_n=1+2^{-n}\). If one recalls that the weight is doubling so that \(w(B_{h_n})\leq \gamma_ww(B_{h_{n+1}})\) we deduce from \eqref{fund-esti2} that (after passing to the limit \(l\to \infty\))
\begin{equation}\label{moser-step0}
\pt{\fint_{B_{h_{n+1}}}\abs{\bar u^\alpha}^{\chi p}w}^{\frac1{\chi p}}\leq C2^n\alpha^{\frac{p}{\ve}}\pt{\fint_{B_{h_n}}\abs{\bar u^\alpha}^pw}^{\frac1p},
\end{equation}
which is valid for all \(\alpha\geq 1\). We observe that using \eqref{moser-step0} for \(\alpha_n=\chi^n\geq 1\) gives
\[
\pt{\fint_{B_{h_{n+1}}}\abs{\bar u}^{s_{n+1}}w}^{\frac1{s_{n+1}}}\leq C^{\chi^{-n}}2^{n\chi^{-n}}\chi^{\frac{p}{\ve}n\chi^{-n}}\pt{\fint_{B_{h_n}}\abs{\bar u}^{s_n}w}^{\frac1{s_n}},
\]
where \(s_n=p\chi^n\). Because \(\chi>1\) then \(\sum_{k=0}^\infty k\chi^{-k}\) and \(\sum_{k=0}^\infty\chi^{-k}\) are convergent series so we can iterate the above inequality to obtain
\[
\pt{\fint_{B_{h_{n+1}}}\abs{\bar u}^{s_{n+1}}w}^{\frac1{s_{n+1}}}\leq C\pt{\fint_{B_{2}}\abs{\bar u}^pw}^{\frac1p},
\]
which after passing to the limit \(n\to\infty\) yields\footnote{Note that \(L^{\infty,w}(B_1)=L^{\infty}(B_1)\) because \(0<w<\infty\) a.e. in \(\Omega\).}
\[
\norm{u}_{L^{\infty}(B_1)}\leq C\spt{\pt{\fint_{B_{2}}\abs{u}^pw}^{\frac1p}+k},
\]
and the result follows in the case \(R=1\).

If \(R\neq 1\) a standard scaling argument allows us to reduce the situation to the case \(R=1\). We include this argument as it will be used a few more times in the following proofs. If we consider \(\tilde u(y)=u(Ry)\) where \(u\) is a weak solution of \(\di\A=\B\) in \(B_{2R}(x_0)\), then \(\tilde u\) is a weak solution of \(\di\tilde \A=\tilde \B\) in \(B_2(y_0)\) where \(y_0=R^{-1}x_0\) and
\[
\tilde\A(y,\tilde u,\tilde z)=R^{p-1}\A(Ry,\tilde u,R^{-1}\tilde z),\quad \tilde\B(y,\tilde u,\tilde z)=R^{p}\B(Ry,\tilde u,R^{-1}\tilde z).
\]
It is clear that if \(\A\) and \(\B\) satisfy \eqref{hyp-A1},\eqref{hyp-A2},\eqref{hyp-B} then \(\tilde \A\) and \(\tilde \B\) satisfy
\begin{gather}
\tag{\(\widetilde{\text{H1}}\)}
\tilde\A(y,\tilde u,\tilde z)\cdot \tilde z\geq \tilde w(y)\pt{\abs{\tilde z}^p-\tilde d\abs{\tilde u}^p-\tilde g},\\
\tag{\(\widetilde{\text{H2}}\)}
\abs{\A(y,\tilde u,\tilde z)}\leq \tilde w(y)\pt{\abs{\tilde z}^{p-1}+\tilde b\abs{\tilde u}^{p-1}+\tilde e},\\
\tag{\(\widetilde{\text{H3}}\)}
\abs{\B(y,\tilde u,\tilde z)}\leq \tilde w(y)\pt{\tilde c\abs{\tilde z}^{p-1}+\tilde d\abs{\tilde u}^{p-1}+\tilde f},
\end{gather}
where
\begin{gather*}
\tilde w(y)=w(Ry),\quad \tilde c(y)=Rc(Ry)\\
\tilde b(y)=R^{p-1}b(Ry),\quad \tilde e(y)=R^{p-1}e(Ry),\\
\tilde d(y)=R^pd(Ry),\quad \tilde f(y)=R^pf(Ry),\quad \tilde g(y)=R^pg(Ry),
\end{gather*}
hence we can apply the case \(R=1\) to \(\tilde u\) to obtain
\[
\norm{u}_{L^{\infty}(B_R(x_0))}=\norm{\tilde u}_{L^{\infty}(B_1(y_0))}\leq C\spt{\pt{\fint_{B_{2}(y_0)}\abs{\tilde u}^p\tilde w}^{\frac1p}+\tilde k},
\]
where
\[
\tilde k=\pt{\pt{\fint_{B_2(y_0)}\abs{\tilde e}^{\frac{D}{p-1}}\tilde w}^{\frac{p-1}{D}}+\pt{\fint_{B_2(y_0)}\abs{\tilde f}^{\frac{D}{p-\ve}}\tilde w}^{\frac{p-\ve}{D}}}^{\frac{1}{p-1}}+\pt{\pt{\fint_{B_2(y_0)}\abs{\tilde g}^{\frac{D}{p-\ve}}\tilde w}^{\frac{p-\ve}{D}}}^{\frac1p},
\]
and the result follows if we observe that
\begin{gather*}
\pt{\fint_{B_2(y_0)}\abs{\tilde b}^{\frac{D}{p-1}}\tilde w}^{\frac{p-1}{D}}=R^{p-1}\pt{\fint_{B_{2R}(x_0)}\abs{b}^{\frac{D}{p-1}}w}^{\frac{p-1}{D}},\\
\pt{\fint_{B_2(y_0)}\abs{\tilde c}^{\frac{D}{1-\ve}}\tilde w}^{\frac{1-\ve}{D}}=R\pt{\fint_{B_{2R}(x_0)}\abs{c}^{\frac{D}{1-\ve}}w}^{\frac{1-\ve}{D}},\\
\pt{\fint_{B_2(y_0)}\abs{\tilde d}^{\frac{D}{p-\ve}}\tilde w}^{\frac{p-\ve}{D}}=R^{p}\pt{\fint_{B_{2R}(x_0)}\abs{d}^{\frac{D}{p-\ve}}w}^{\frac{p-\ve}{D}},\\
\pt{\fint_{B_2(y_0)}\abs{\tilde e}^{\frac{D}{p-1}}\tilde w}^{\frac{p-1}{D}}=R^{p-1}\pt{\fint_{B_{2R}(x_0)}\abs{e}^{\frac{D}{p-1}}w}^{\frac{p-1}{D}},\\
\pt{\fint_{B_2(y_0)}\abs{\tilde f}^{\frac{D}{p-\ve}}\tilde w}^{\frac{p-\ve}{D}}=R^{p}\pt{\fint_{B_{2R}(x_0)}\abs{f}^{\frac{D}{p-\ve}}w}^{\frac{p-\ve}{D},}\\
\pt{\fint_{B_2(y_0)}\abs{\tilde g}^{\frac{D}{p-\ve}}\tilde w}^{\frac{p-\ve}{D}}=R^{p}\pt{\fint_{B_{2R}(x_0)}\abs{g}^{\frac{D}{p-\ve}}w}^{\frac{p-\ve}{D}}.
\end{gather*}
\end{proof}

\begin{remark}\label{loc-glob-rem1}
If the global Sobolev inequality \eqref{sob-ineq-w} is satisfied then there is no need to average the integrals in the above proof. For instance, to obtain \eqref{gath4} we used Hölder's inequality to write
\[
\fint_{B_2}\bar d\abs{v\eta}^pw\leq
\pt{\fint_{B_2}\bar d^{\frac{D}{p-\ve}}w}^{\frac{p-\ve}{D}}
\pt{\fint_{B_2}\abs{v\eta}^pw}^{\frac{\ve}{p}}
\pt{\fint_{B_2}\abs{v\eta}^{\chi p}w}^{\frac{p-\ve}{\chi p}}
\]
and then we used the local Sobolev inequality \eqref{local-Sobolev} to estimate the last term as
\[
\pt{\fint_{B_2}\abs{v\eta}^{\chi p}w}^{\frac{p-\ve}{\chi p}}\leq C\pt{\fint_{B_2}\abs{\nabla(v\eta)}^{p}w}^{\frac{p-\ve}{p}}\leq C\spt{
	\pt{\fint_{B_2}\abs{\eta\nabla v}^{p}w}^{\frac{p-\ve}{p}}
	+
	\pt{\fint_{B_2}\abs{v\nabla\eta}^{p}w}^{\frac{p-\ve}{p}}
	}.
\]
However, if the global Sobolev inequality holds then the same follows without averaging, that is we would have
\[
\int_{B_2}\bar d\abs{v\eta}^pw
\leq C\pt{\int_{B_2}\abs{v\eta}^{p}w}^{\frac{\ve}{p}}\pt{\pt{\int_{B_2}\abs{v\nabla\eta}^{p}w}^{\frac{p-\ve}{p}}+\pt{\int_{B_2}\abs{\eta\nabla v}^{p}w}^{\frac{p-\ve}{p}}}
\]
and similarly for \eqref{gath1}-\eqref{gath3} and \eqref{fund-esti2}. In particular when the global Sobolev inequality holds we have a version of \cref{thm-local-bdd} where the following estimate holds
\[
\norm{u}_{L^{\infty}(B_R(x_0))}\leq C_R\spt{\norm{u}_{L^{p,w}(B_{2R}(x_0))}+k_R},
\]
and the integrals defining \(k_R\) are also not averaged.
\end{remark}

\begin{proof}[Proof of \cref{thm-local-inte}]
It is enough to consider the case \(R=1\) because the scaling argument remains the same. Additionally, thanks to the interpolation inequality in \(L^{s,w}\), it is enough to find a sequence \(s_n\sublim\limits_{n\to\infty}+\infty\) for which one has
\[
\pt{\fint_{B_1}\abs{\bar u}^{s_n}w}^{\frac1{s_n}}\leq C_n\pt{\fint_{B_{2}}\abs{\bar u}^pw}^{\frac1p},
\]
where \(\bar u=\abs{u}+k\). As in the proof of \cref{thm-local-bdd}, by using the test function \(\varphi=\eta^pG(u)\) we reach to the inequality
\begin{multline*}
\fint_{B_2}\abs{\eta\nabla v}^pw\leq p\fint_{B_2}\abs{v\nabla \eta}\abs{\eta\nabla v}^{p-1}w+p\alpha^{p-1}\fint_{B_2}\bar b\abs{v\nabla \eta}\abs{v\eta}^{p-1}w\\+\fint_{B_2} cv\eta\abs{\eta\nabla v}^{p-1}w+(1+\beta)\alpha^{p-1}\fint_{B_2}\bar d\abs{v\eta}^pw,
\end{multline*}
but because \(\ve=0\) we cannot repeat \eqref{gath3}-\eqref{gath4}. For the terms involving \(c\) and \(\bar d\) we can write for \(M>0\) the following
\begin{align*}
\fint_{B_2} cv\eta\abs{\eta\nabla v}^{p-1}w&=\frac{1}{w(B_2)}\pt{\int_{B_2\cap\set{c\leq M}} cv\eta\abs{\eta\nabla v}^{p-1}w+\int_{B_2\cap \set{c>M}} cv\eta\abs{\eta\nabla v}^{p-1}w}\\
&\leq M\pt{\fint_{B_2} \abs{v\eta}^pw}^{\frac1p}\pt{\fint_{B_2} \abs{\eta\nabla v}^pw}^{1-\frac1p}\\
&\quad+\pt{\frac{1}{w(B_2)}\int_{B_2\cap \set{c>M}}c^{D}w}^{\frac1D}\pt{\fint_{B_2} \abs{v\eta}^{\chi p}w}^{\frac1{\chi p}}\pt{\fint_{B_2} \abs{\eta\nabla v}^pw}^{1-\frac1p}\\
&\leq M\pt{\fint_{B_2} \abs{v\eta}^pw}^{\frac1p}\pt{\fint_{B_2} \abs{\eta\nabla v}^pw}^{1-\frac1p}\\
&\quad+C\pt{\fint_{B_2}\abs{v\nabla\eta}^pw}^{\frac1p}\pt{\fint_{B_2} \abs{\eta\nabla v}^pw}^{1-\frac1p}\\
&\quad+C\pt{\frac{1}{w(B_2)}\int_{B_2\cap \set{c>M}}c^{D}w}^{\frac1D}\pt{\fint_{B_2}\abs{\eta \nabla v}^pw}
\end{align*}
and
\begin{align*}
\fint_{B_2}\bar d\abs{v\eta}^pw&=\frac{1}{w(B_2)}\pt{\int_{B_2\cap\set{\bar d\leq M}}\bar d\abs{v\eta}^pw+\int_{B_2\cap\set{\bar d>M}}\bar d\abs{v\eta}^pw}\\
&\leq M\fint_{B_2}\abs{v\eta}^pw+\pt{\frac{1}{w(B_2)}\int_{B_2\cap\set{\bar d>M}}\bar d^{\frac{D}{p}}w}^{\frac{p}D}\pt{\fint_{B_2}\abs{v\eta}^{\chi p}w}^{\frac1\chi}\\
&\leq M\fint_{B_2}\abs{v\eta}^pw+C\pt{\fint_{B_2}\abs{v \nabla\eta}^{p}w}\\
&\quad+\pt{\frac{1}{w(B_2)}\int_{B_2\cap\set{\bar d>M}}\bar d^{\frac{D}{p}}w}^{\frac{p}D}\pt{\fint_{B_2}\abs{\eta \nabla v}^{p}w}.
\end{align*}

Because \(c\in L^{D,w}\) and \(\bar d\in L^{\frac{D}{p},w}\) then for any \(\delta>0\) we can find \(M>0\) such that
\[
C\pt{\frac{1}{w(B_2)}\int_{B_2\cap \set{c>M}}c^{D}w}^{\frac1D}+\pt{\frac{1}{w(B_2)}\int_{B_2\cap\set{\bar d>M}}\bar d^{\frac{D}{p}}w}^{\frac{p}D}\leq \delta,
\]
therefore for any \(\alpha\geq 1\) we can find \(\delta>0\) sufficiently small and a constant \(C_\alpha>0\) such that
\begin{multline*}
\fint_{B_2}\abs{\eta\nabla v}^pw\leq
C_\alpha
\spt{\pt{\fint_{B_2}\abs{v\nabla \eta}^pw}^{\frac1p}+\pt{\fint_{B_2}\abs{v\eta}^pw}^{\frac1p}}\pt{\fint_{B_2}\abs{\eta\nabla v}^{p}w}^{1-\frac1p}\\
+C_\alpha\pt{\fint_{B_2}\abs{v\nabla \eta}^pw}+C_\alpha\pt{\fint_{B_2}\abs{v\eta}^pw}.
\end{multline*}
The above inequality allows us to use \cite[Lemma 2]{Serrin1964} once again and obtain an inequality analogous to \eqref{fund-esti}, namely
\begin{equation}\label{fund-esti4}
\pt{\fint_{B_2}\abs{\eta \nabla v}^pw}^{\frac1p}\leq C_\alpha\pt{\pt{\fint_{B_2}\abs{\eta v}^pw}^{\frac1p}+\pt{\fint_{B_2}\abs{v\nabla \eta}^pw}^{\frac1p}}
\end{equation}
the main difference being that the constant \(C_\alpha\) is no longer explicit. Nonetheless we can continue the argument from the proof of \cref{thm-local-bdd} by choosing appropriate cut-off functions \(\eta\) to reach
\[
\pt{\fint_{B_{h_{n}}}\abs{\bar u}^{s_{n}}w}^{\frac1{s_{n}}}\leq C_n\pt{\fint_{B_{2}}\abs{\bar u}^pw}^{\frac1p},
\]
where \(s_n=p\chi^n\) and \(h_n=1+2^{-n}\), where now we do not obtain a uniform estimate for \(C_n\). Finally, because \(1\leq h\leq 2\) we know that
\[
\frac{w(B_h)}{w(B_1)}\leq\gamma_w
\]
so that
\[
\pt{\fint_{B_{1}}\abs{\bar u}^{s_{n}}w}^{\frac1{s_{n}}}\leq C_n\pt{\fint_{B_{2}}\abs{\bar u}^pw}^{\frac1p}
\]
and the result is proved.
\end{proof}

\begin{remark}\label{loc-glob-rem2}
Similar to \cref{loc-glob-rem1}, if the global Sobolev inequality holds then there is no need to average the integrals in the above proof and as a consequence we obtain the following estimate in \cref{thm-local-inte}
\[
\norm{u}_{L^{s,w}(B_R(x_0))}\leq C_{R,s}\spt{\norm{u}_{L^{p,w}(B_{2R}(x_0))}+k_{R}},
\]
where once again \(k_R\) needs to be changed to the non-averaged version.
\end{remark}

\begin{proof}[Proof of \cref{thm-local-lower}]
This is similar to the proof of \cref{thm-local-bdd,thm-local-inte} and \cite[Theorem 1']{Serrin1964}. As before, we only do the case \(R=1\). For \(F\) and \(G\) as in \eqref{func-F} and \eqref{func-G} respectively we take \(\varphi=\eta^p G(u)\) as test function for some \(\eta\in C^\infty_c(B_2)\) with \(0\leq \eta\leq 1\). The main difference here is that we take \(\bar u=\abs{u}\) as we need to keep track of the terms involving \(e,f,g\) instead of including them in \(\bar b\) and \(\bar d\) through the use of \(k>0\). With that in mind and using that \(F'(\bar u)\leq \alpha F(\bar u)^{1-\frac{1}{\alpha}}\) we have
\begin{align*}
\A\cdot\nabla \varphi+\B\varphi
&\geq \eta^p G'(u)w\pt{\abs{\nabla u}^{p}-du^{p}-g}
-p\eta^{p-1}\abs{\nabla \eta G(u)}w\pt{\abs{\nabla u}^{p-1}+bu^{p-1}+e}\\
&\quad -\eta^p\abs{G(u)}w\pt{c\abs{\nabla u}^{p-1}+du^{p-1}+f}\\
&\geq \abs{\eta F'(\bar u)\nabla u}^{p}w
-p\abs{\nabla \eta F(\bar u)}\abs{\eta F'(\bar u)\nabla u}^{p-1}w\\
&\quad -p\alpha^{p-1}b\abs{F(\bar u)\nabla \eta }\abs{\eta F(\bar u)}^{p-1}w-c\abs{\eta F(\bar u)}\abs{\eta F'(\bar u)\nabla u}^{p-1}w\\
&\quad -(1+\beta)\alpha^{p-1}d\abs{\eta F(\bar u)}^{p}w-p\alpha^{p-1}e\abs{F(\bar u)\nabla \eta}\abs{\eta F(\bar u)}^{\frac{\beta}{\alpha}-1}w\\
&\quad -\alpha^{p-1}f\abs{\eta F(\bar u)}\abs{\eta F(\bar u)}^{\frac{\beta}{\alpha}}w
-\beta\alpha^{p-1}g\abs{\eta F(\bar u)}^{p\frac{\alpha-1}{\alpha}}w
\end{align*}
which after averaging over \(B_2\) gives
\begin{multline}\label{bas-esti5}
\fint_{B_2}\abs{\eta\nabla v}^pw\leq
p\fint_{B_2}\abs{v\nabla \eta}\abs{\eta\nabla v}^{p-1}w
+p\alpha^{p-1}\fint_{B_2}b\abs{v\nabla \eta}\abs{v\eta}^{p-1}w\\
+\fint_{B_2} cv\eta\abs{\eta\nabla v}^{p-1}w
+(1+\beta)\alpha^{p-1}\fint_{B_2}d\abs{v\eta}^pw
+p\alpha^{p-1}\fint_{B_2}e\abs{v\nabla \eta}\abs{\eta v}^{\frac{\beta}{\alpha}-1}w\\
+\alpha^{p-1}\fint_{B_2}f\abs{\eta v}^{\frac{\beta}{\alpha}}w
+\beta\alpha^{p-1}\fint_{B_2}g\abs{\eta v}^{\frac{p(\alpha-1)}{\alpha}}w,
\end{multline}
where \(v=F(\bar u)\).

With an inequality similar to \eqref{fund-esti} or \eqref{fund-esti4} as our goal we estimate each term in the above inequality. The terms involving \(b,c,d\) are bounded exactly as in \cref{thm-local-inte}, that is we obtain
\begin{multline*}
p\fint_{B_2}\abs{v\nabla \eta}\abs{\eta\nabla v}^{p-1}w
+p\alpha^{p-1}\fint_{B_2}b\abs{v\nabla \eta}\abs{v\eta}^{p-1}w
+\fint_{B_2} cv\eta\abs{\eta\nabla v}^{p-1}w\\
+(1+\beta)\alpha^{p-1}\fint_{B_2}d\abs{v\eta}^pw
\leq 
C_\alpha
\spt{\pt{\fint_{B_2}\abs{v\nabla \eta}^pw}^{\frac1p}+\pt{\fint_{B_2}\abs{v\eta}^pw}^{\frac1p}}\pt{\fint_{B_2}\abs{\eta\nabla v}^{p}w}^{1-\frac1p}\\
+C_\alpha\pt{\fint_{B_2}\abs{v\nabla \eta}^pw}+C_\alpha\pt{\fint_{B_2}\abs{v\eta}^pw},
\end{multline*}
for some constant \(C_\alpha>0\). 

For the other terms let \(s_1,s_2,s_3\geq 1\), and consider \(\alpha\geq1\) such that \(\frac{1}{s_1}+\frac1p+\frac{\beta-\alpha}{\alpha\chi p}=\frac1{s_2}+\frac{\beta}{\alpha\chi p}=1\) then for any \(\delta>0\) small we can use Hölder's inequality and \eqref{local-Sobolev} to obtain a constant \(C_\delta>0\) such that
\begin{align*}
\fint_{B_2}e\abs{v\nabla \eta}\abs{\eta v}^{\frac{\beta}{\alpha}-1}w&\leq \pt{\fint_{B_2}\abs{e}^{s_1}w}^{\frac1{s_1}}\pt{\fint_{B_2}\abs{v\nabla \eta}^pw}^{\frac1p}\pt{\fint_{B_2}\abs{\eta v}^{\chi p}w}^{\frac{\beta-\alpha}{\alpha\chi p}}\\
&\leq C\pt{\fint_{B_2}\abs{e}^{s_1}w}^{\frac1{s_1}}\pt{\fint_{B_2}\abs{v\nabla \eta}^pw}^{\frac1p}\pt{\fint_{B_2}\abs{\nabla(\eta v)}^{p}w}^{\frac{\beta-\alpha}{\alpha p}}\\
&\leq C_\delta\pt{\pt{\fint_{B_2}\abs{e}^{s_1}w}^{\frac{p\alpha}{s_1(p-1)}}+\pt{\fint_{B_2}\abs{v\nabla \eta}^pw}}\\
&\quad+\delta\pt{\fint_{B_2}\abs{\nabla(\eta v)}^{p}w}.
\end{align*}
Similarly
\begin{align*}
\fint_{B_2}f\abs{\eta v}^{\frac{\beta}{\alpha}}w&\leq\pt{\fint_{B_2}\abs{f}^{s_2}w}^{\frac1{s_2}}\pt{\fint_{B_2}\abs{\eta v}^{\chi p}w}^{\frac{\beta}{\alpha \chi p}}\\
&\leq C\pt{\fint_{B_2}\abs{f}^{s_2}w}^{\frac1{s_2}}\pt{\fint_{B_2}\abs{\nabla(\eta v)}^{p}w}^{\frac{\beta}{\alpha p}}\\
&\leq C_\delta\pt{\pt{\fint_{B_2}\abs{f}^{s_2}w}^{\frac{p\alpha}{s_2(p-1)}}+\pt{\fint_{B_2}\abs{v\nabla\eta}^{p}w}}+\delta\pt{\fint_{B_2}\abs{\eta \nabla v}^{p}w},
\end{align*}
and if \(\frac1{s_3}+\frac{p(\alpha-1)}{\alpha \chi p}=1\)
\begin{align*}
\fint_{B_2}g\abs{\eta v}^{\frac{p(\alpha-1)}{\alpha}}w&\leq \pt{\fint_{B_2}\abs{g}^{s_3}w}^{\frac1{s_3}}\pt{\fint_{B_2}\abs{\eta v}^{\chi p}w}^{\frac{p(\alpha-1)}{\alpha \chi p}}\\
&\leq C\pt{\fint_{B_2}\abs{g}^{s_3}w}^{\frac1{s_3}}\pt{\fint_{B_2}\abs{\nabla(\eta v)}^{p}w}^{\frac{p(\alpha-1)}{\alpha p}}\\
&\leq C_\delta\pt{\pt{\fint_{B_2}\abs{g}^{s_3}w}^{\frac{\alpha}{s_3}}+\pt{\fint_{B_2}\abs{v\nabla\eta }^{p}w}}+\delta\pt{\fint_{B_2}\abs{\eta\nabla v}^{p}w}.
\end{align*}
If we put the above estimates in \eqref{bas-esti5} for \(\delta>0\) small enough we obtain a constant \(C_{\alpha}>0\) such that
\begin{align*}
\fint_{B_2}\abs{\eta\nabla v}^pw
&\leq
C_\alpha
\spt{\pt{\fint_{B_2}\abs{v\nabla \eta}^pw}^{\frac1p}+\pt{\fint_{B_2}\abs{v\eta}^pw}^{\frac1p}}\pt{\fint_{B_2}\abs{\eta\nabla v}^{p}w}^{1-\frac1p}\\
&\quad+C_\alpha\spt{\pt{\fint_{B_2}\abs{v\nabla \eta}^pw}+\pt{\fint_{B_2}\abs{v\eta}^pw}}\\
&\quad+C_\alpha\spt{\pt{\fint_{B_2}\abs{e}^{s_1}w}^{\frac{p\alpha}{s_1(p-1)}}+\pt{\fint_{B_2}\abs{f}^{s_2}w}^{\frac{p\alpha}{s_2(p-1)}}+\pt{\fint_{B_2}\abs{g}^{s_3}w}^{\frac{\alpha}{s_3}}}
\end{align*}
which with the aid of \cite[Lemma 2]{Serrin1964} one more time gives
\[
\pt{\fint_{B_2}\abs{\eta\nabla v}^pw}^{\frac1p}
\leq
C_\alpha \spt{\pt{\fint_{B_2}\abs{\eta\nabla v}^pw}^{\frac1p}+\pt{\fint_{B_2}\abs{\eta v}^pw}^{\frac1p}+M^\alpha}
\]
where
\[
M=\pt{\fint_{B_2}\abs{e}^{s_1}w}^{\frac{p}{s_1(p-1)}}+\pt{\fint_{B_2}\abs{f}^{s_2}w}^{\frac{p}{s_2(p-1)}}+\pt{\fint_{B_2}\abs{g}^{s_3}w}^{\frac{1}{s_3}}.
\]

We can proceed as in the previous theorems by using \eqref{local-Sobolev}, selecting \(\eta\) and passing to the limit \(l\to\infty\) to obtain
\begin{equation}\label{bas-ineq38}
\pt{\fint_{B_{h_{n+1}}}\abs{u}^{p\alpha\chi}w}^{\frac1{p\alpha\chi}}\leq C_{\alpha,n}\spt{\pt{\fint_{B_{h_n}}\abs{u}^{p\alpha}w}^{\frac1{p\alpha}}+M},
\end{equation}
where \(h_n=1+2^{-n}\) and \(\alpha\geq1\) is chosen so that
\begin{align}
\label{cond1}\frac{1}{s_1}+\frac1p+\frac{\beta-\alpha}{\alpha\chi p}=\frac1{s_2}+\frac{\beta}{\alpha\chi p}=1\\
\label{cond2}\frac1{s_3}+\frac{p(\alpha-1)}{\alpha \chi p}=1.
\end{align}
Observe that \eqref{cond1} implies that \(s_1=\chi s_2\) and as a consequence \(M\) is finite if \(s_2\leq r\) and \(s_3\leq t\). Therefore we can iterate \eqref{bas-ineq38} by selecting \(\alpha=\alpha_n\geq 1\) satisfying \eqref{cond1}-\eqref{cond2}. Recalling that \(\chi=\frac{D}{D-p}\) we note that the iteration can be done provided
\begin{align*}
s_2\leq r \text{ and } \eqref{cond1}&:\ \frac{1}{p\alpha\chi}\geq \frac{1}{p-1}\pt{\frac1r-\frac{p}{D}}=\frac1{s}\\
s_3\leq t \text{ and } \eqref{cond2}&:\ \frac{1}{p\alpha\chi}\geq \frac{1}{p}\pt{\frac1t-\frac{p}{D}}=\frac1{s}
\end{align*}
which mean that after a finite number of steps we will obtain
\[
\pt{\fint_{B_{1}}\abs{u}^{s}w}^{\frac1{s}}\leq C_{\alpha,n_0}\spt{\pt{\fint_{B_{2}}\abs{u}^{p}w}^{\frac1{p}}+M},
\]
for some \(n_0\in\NN\).
\end{proof}

\begin{proof}[Proof of \cref{harnack-thm}]
As before, we only consider the case \(R=1\). \cref{thm-local-bdd} says that \(u\) is bounded on any compact subset of \(B_{3}\) hence for any \(\beta\in\RR\) and any \(\delta>0\) the function \(\varphi=\eta^p\bar u^{\beta}\) is a valid test function provided \(\bar u=u+k+\delta\) and \(\eta\in C^\infty_c(B_3)\). Here \(k\) is defined exactly as in \cref{thm-local-bdd}.

We begin with the case which differs the most from the proof of \cite[Theorem 5]{Serrin1964}. Suppose that \(v=\log \bar u\) and by following the idea of the proof in \cite{Serrin1964} we reach to
\begin{multline}\label{esti-jn-1}
(p-1)\int_{B_3} \abs{\eta\nabla v}^pw\leq p\int_{B_3} \abs{\nabla \eta}\abs{\eta\nabla v}^{p-1}w+p\int_{B_3} \bar b\eta^{p-1}\abs{\nabla\eta}w+\int_{B_3} c\eta\abs{\eta\nabla v}^{p-1}w\\
+p\int_{B_3} \bar d\eta^pw,
\end{multline}
for any \(\eta\in C^\infty_c(B_3)\). To continue denote by \(z=\pt{\int_{B_3}\abs{\eta\nabla v}^pw}^{\frac1p}\) and with the aid of Hölder's inequality \eqref{esti-jn-1} becomes
\[
z^{p}\leq C_1z^{p-1}+C_2,
\]
where
\begin{gather*}
C_1=\frac{p}{p-1}\pt{\int_{B_3}\abs{\nabla\eta}^pw}^{\frac1p}+\frac1{p-1}\pt{\int_{B_3}\abs{c\eta}^pw}^{\frac1p},\\
C_2=\frac{p}{p-1}\int_{B_3}\bar b\eta^{p-1}\abs{\nabla\eta}w+\frac{p}{p-1}\int_{B_3}\bar d\eta^pw,
\end{gather*}
which thanks to Young's inequality implies
\[
z^p\leq C(C_1^p+C_2),
\]
for some constant \(C\). To continue we estimate \(C_1\) and \(C_2\) using appropriate \(\eta\). For any \(0<h<2\) such that \(B_h\subset B_2\) (not necessarily concentric) we have that \(B_{\frac{3h}2}\subset B_3\) and we consider \(\eta\in C^\infty_c(B_{\frac{3h}2})\) such that \(\eta\equiv 1\) in \(B_h\) and \(\abs{\nabla \eta}\leq Ch^{-1}\). We use such \(\eta\) in \eqref{esti-jn-1} and we perform the following estimates
\begin{equation*}
\int_{B_3} \abs{\nabla \eta}^pw\leq \frac{C}{h^p}w(B_{\frac{3h}2}),
\end{equation*}
\begin{align*}
\int_{B_3} \bar b\eta^{p-1}\abs{\nabla\eta}w&\leq \frac{Cw(B_{\frac{3h}{2}})}{h}\pt{\frac{w(B_3)}{w(B_{\frac{3h}{2}})}}^{\frac{p-1}{D}}\pt{\fint_{B_3} \abs{\bar b}^{\frac{D}{p-1}}w}^{\frac{p-1}{D}}\\
&\leq \frac{Cw(B_{\frac{3h}{2}})}{h^{p}}\pt{\fint_{B_3} \abs{\bar b}^{\frac{D}{p-1}}w}^{\frac{p-1}{D}},
\end{align*}
\begin{align*}
\int_{B_3} \abs{c\eta}^pw&\leq Cw(B_{\frac{3h}2})\pt{\frac{w(B_3)}{w(B_{\frac{3h}{2}})}}^{\frac{(1-\ve)p}{D}}\pt{\fint_{B_3} \abs{c}^{\frac{D}{1-\ve}}w}^{\frac{(1-\ve)p}{D}}\\
&\leq C\frac{w(B_{\frac{3h}2})}{h^{(1-\ve)p}}\pt{\fint_{B_3} \abs{c}^{\frac{D}{1-\ve}}w}^{\frac{1-\ve}{D}},
\end{align*}
\begin{align*}
\int_{B_3} \bar d\eta^pw&\leq Cw(B_{\frac{3h}2})\pt{\frac{w(B_3)}{w(B_{\frac{3h}{2}})}}^{\frac{p-\ve}{D}}\pt{\fint_{B_3} \abs{\bar d}^{\frac{D}{p-\ve}}w}^{\frac{p-\ve}{D}}\\
&\leq C\frac{w(B_{\frac{3h}2})}{h^{p-\ve}}\pt{\fint_{B_3} \abs{\bar d}^{\frac{D}{p-\ve}}w}^{\frac{p-\ve}{D}},
\end{align*}
where we have used \eqref{w-ball-estimate} repeatedly. Therefore one obtains
\begin{align*}
\fint_{B_h}\abs{\nabla v}^pw&\leq \frac{1}{w(B_h)}\int_{B_3}\abs{\eta\nabla v}^pw\\
&\leq \frac{1}{w(B_h)}\pt{C_1^p+C_2}\\
&\leq C\frac{w(B_{\frac{3h}{2}})}{w(B_h)}\pt{h^{-p}+h^{-(1-\ve)p}+h^{-(p-\ve)}}\\
&\leq C\pt{h^{-p}+h^{-(1-\ve)p}+h^{-(p-\ve)}},
\end{align*}
where now \(C\) depends on \(\fint_{B_3}\abs{\bar b}^{\frac{D}{p-1}}w,\ \fint_{B_3}\abs{c}^{\frac{D}{1-\ve}}w\) and \(\fint_{B_3}\abs{\bar d}^{\frac{D}{p-\ve}}w\).

Finally, the local Poincaré-Sobolev inequality \eqref{local-Poincare} tells us that
\begin{align*}
\fint_{B_h}\abs{v-v_{B_h}}w&\leq \pt{\fint_{B_h}\abs{v-v_{B_h}}^pw}^{\frac1p}\\
&\leq Ch\pt{\fint_{B_h}\abs{\nabla v}^pw}^{\frac1p}\\
&\leq C\pt{1+h^{p\ve}+h^{\ve}}^{\frac1p},
\end{align*}
and because \(h\leq 2\) for any ball \(B_h\subseteq B_2\) we conclude that
\[
\fint_{B_h}\abs{v-v_{B_h}}w\leq C
\]
where \(C>0\) is a constant not depending on \(h\), in other words, \(v\in \mathrm{BMO}(B_2,w\dx)\) and the John-Nirenberg lemma for doubling measures \cite[Appendix II]{HeKiMa2006} tells us that there exists constants \(p_0, C>0\) such that
\[
\fint_{B}e^{p_0\abs{v-v_{B}}}w\leq C
\]
for all balls \(B\subseteq B_2\), in particular this gives
\[
\pt{\fint_{B_2}e^{p_0v}w} \cdot \pt{\fint_{B_2}e^{-p_0v}w} \leq C^2,
\]
and because \(v=\log \bar u\) we have obtained
\[
\fint_{B_2}\bar u^{p_0}w\leq C\pt{\fint_{B_2}\bar u^{-p_0}w}^{-1}.
\]
If we denote by \(\Psi(p,h)=\pt{\fint_{B_h}\bar u^pw}^{\frac1p}\) then the above becomes
\begin{equation}\label{esti-37}
\Psi(p_0,2)\leq C\Psi(-p_0,2).
\end{equation}

The rest of the proof consists in using \(\varphi=\eta^p\bar u^\beta\) for \(\beta\neq 1-p,0\) as test function. If \(v=\bar u^\alpha\) for \(\alpha\) given by \(p\alpha=p+\beta-1\) then following the ideas from \cite{Serrin1964} leads to
\begin{itemize}[leftmargin=*]
\item If \(\beta>0\) then
\begin{equation*}
\pt{\fint_{B_3} \abs{\eta v}^{\chi p}w}^{\frac1{\chi p}}\leq C\alpha^{\frac{p}{\ve}}(1+\beta^{-1})^{\frac1\ve}\spt{\pt{\fint_{B_3} \abs{\eta v}^pw}^{\frac1p}+\pt{\fint_{B_3} \abs{\nabla\eta v}^pw}^{\frac1p}}.
\end{equation*}
If \(\eta\in C^\infty_c(B_h)\) is such that \(\eta\equiv 1\) in \(B_{h'}\) for \(1\leq h'<h\leq 2\) with \(\abs{\nabla \eta}\leq C(h-h')^{-1}\) then
\begin{equation*}
\pt{\fint_{B_{h'}} \abs{v}^{\chi p}w}^{\frac1{\chi p}}\leq C\pt{\frac{w(B_3)}{w(B_{h'})}}^{\frac1{\chi p}}\pt{\frac{w(B_{h})}{w(B_3)}}^{\frac1{p}}\frac{\alpha^{\frac{p}{\ve}}(1+\beta^{-1})^{\frac1\ve}}{h-h'}\pt{\fint_{B_h} \abs{v}^pw}^{\frac1p},
\end{equation*}
but since \(1\leq h'<h\leq 2\) we have
\[
\frac{w(B_3)}{w(B_{h'})}\leq \frac{w(B_{4h'})}{w(B_{h'})}\leq \gamma_w^2\quad \text{and}\quad \frac{w(B_h)}{w(B_{3})}\leq 1
\]
hence
\begin{equation}\label{esti-38}
\Psi(\chi p,h')
\leq C\frac{\alpha^{\frac{p}{\ve}}(1+\beta^{-1})^{\frac1\ve}}{h-h'}\Psi(p,h).
\end{equation}

\item Similarly, for \(1-p<\beta<0\) one has
\begin{equation}\label{esti-39}
\Psi(\chi p,h')
\leq C\frac{(1-\beta^{-1})^{\frac1\ve}}{h-h'}\Psi(p,h).
\end{equation}
\item If \(\beta<1-p\) then one obtains
\begin{equation}\label{esti-40}
\Psi(\chi p',h')
\leq C\frac{(1+\abs{\alpha})^{\frac{p}\ve}}{h-h'}\Psi(p,h).
\end{equation}
\end{itemize}

If we observe that \(\Psi(s,r)\sublim\limits_{s\to\infty} \max\limits_{B_r}\bar u\) and \(\Psi(s,r)\sublim\limits_{s\to-\infty} \min\limits_{B_r}\bar u\) we can repeat the iterative argument from the proof of \cite[Theorem 5]{Serrin1964} to deduce that \eqref{esti-38} and \eqref{esti-39} imply
\[
\max_{B_1}\bar u\leq C\Psi(p_0',2)
\]
for some \(p_0'\leq p_0\) chosen appropriately, whereas \eqref{esti-40} will give
\[
\min_{B_1}\bar u\geq C^{-1}\Psi(-p_0,2).
\]
Finally we can use \eqref{esti-37} to obtain a constant \(C>0\) depending on the structural parameters such that
\[
\max_{B_1}\bar u\leq C\min_{B_1}\bar u
\]
and because \(\bar u=u+k+\delta\) we conclude by letting \(\delta\to 0^+\).
\end{proof}

\begin{remark}\label{k_R-estim}
Observe that \(k_R\) in \cref{thm-local-bdd} can be further estimated as follows: if one supposes that solutions are defined in \(B_{R_0}\) for some \(R_0>0\) then because of \eqref{w-ball-estimate} we can write
\[
\pt{\fint_{B_{2R}}\abs{h}^{s}w}^{\frac1s}\leq \pt{\frac{w(B_{R_0})}{w(B_{2R})}}^{\frac1s}\pt{\fint_{B_{R_0}}\abs{h}^{s}w}^{\frac1s}\leq CR_0^{\frac{D}s}R^{-\frac{D}s}\pt{\fint_{B_{R_0}}\abs{h}^{s}w}^{\frac1s}
\]
for any \(s>0\) and \(0<2R<R_0\) then one can give a better estimate on \(C_R\) and \(k_R\) in terms of \(R>0\) because for example one can see that if \(b\in L^{\frac{D}{p-1-\ve},w}(\Omega)\) then 
\[
b_R:=R^{p-1}\pt{\fint_{B_{2R(x_0)}}\abs{b}^{\frac{D}{p-1-\ve}}w}^{\frac{p-1-\ve}{D}}\leq CR_0^{p-1-\ve}R^\ve\pt{\fint_{B_{R_0}}\abs{b}^{\frac{D}{p-1}}w}^{\frac{p-1}D}\leq C_{R_0,b}R^\ve,
\]
similarly we obtain that 
\[
f_R\leq C_{R_0,f}R^{\ve},\quad
g_R\leq C_{R_0,g}R^{\ve},
\]
so that if
\[
k_0=\pt{\pt{\fint_{B_{R_0}}\abs{e}^{\frac{D}{p-1-\ve}}w}^{\frac{p-1-\ve}{D}}+\pt{\fint_{B_{R_0}}\abs{f}^{\frac{D}{p-\ve}}w}^{\frac{p-\ve}{D}}}^{\frac{1}{p-1}}+\pt{\fint_{B_{R_0}}\abs{g}^{\frac{D}{p-\ve}}w}^{\frac{p-\ve}{Dp}}
\]
then \(k_R\leq k_0\max\set{R^{\frac{\ve}{p-1}},R^{\frac{\ve}{p}}}\). Therefore if \(R_0<1\) we have that
\[
k_R\leq k_0R^{\frac{\ve}{p}}.
\]
\end{remark}

\begin{proof}[Proof of \cref{holder-thm}]
This is standard once we know the Harnack inequality is valid. We just highlight the main steps. For \(x_0\in \Omega\) by a suitable scaling we can suppose that \(B_2=B_2(x_0)\subseteq\Omega\) and we can consider for \(0<r\leq 1\) the functions
\[
M(r)=\max_{B_r}u\quad m(r)=\min_{B_r}u
\]
If we define \(u_M=M-u\) then \(u_M\) is a weak solution of an equation of the form
\[
\di \tilde \A=\tilde \B\quad \text{in }B_2
\]
for suitable \(\tilde{\A}\) and \(\tilde\B\) which satisfy the same structural hypotheses as \(\A\) and \(\B\) because \(u\) is bounded in \(B_1\). The only thing to keep in mind is that \(\tilde b,\tilde c,\ldots,\tilde g\) also depend on \(\max_{B_1}\abs{u}\). The Harnack inequality then implies that
\[
M(r)-m\pt{\frac{r}{3}}\leq C\max_{B_{\frac{r}{3}}} u_M\leq C\pt{\min_{B_{\frac{r}3}}u_M+\tilde k}=\tilde C\pt{M(r)-M\pt{\frac{r}{3}}+\tilde k}.
\]
The same situation occurs for the function \(u_m=u-m(r)\) as it can be shown that
\[
M\pt{\frac{r}{3}}-m(r)=\max_{B_{\frac{r}3}}u_m\leq \tilde C\pt{\min_{B_{\frac{r}3}}u_m+\tilde k}=C\pt{m\pt{\frac{r}{3}}-m(r)+\tilde k}
\]
and if \(\omega(r)=M(r)-m(r)\) denotes the oscillation of \(u\) in \(B_r\) we are led to
\[
\omega\pt{\frac{r}{3}}\leq \frac{\tilde C-1}{\tilde C+1}\omega(r)+\frac{2\tilde C}{\tilde C+1}\tilde k,
\]
but \cref{k_R-estim} tells us that if \(r<1\) then \(\tilde k=\tilde k_r\leq \tilde k_0r^{\frac{\ve}{p}}\) for some \(\tilde k_0\) depending solely on \(b,c,\ldots,g\) and \(\max_{B_1}\abs{u}\). Also by increasing its value one can suppose that \(1\leq \tilde C=\tilde C_r\leq C_0\) so that we have
\[
\omega\pt{\frac{r}{3}}\leq \theta\pt{w(r)+\tau r^{\frac{\ve}{p}}}
\]
for constants \(\theta,\tau>0\). The rest of the argument to conclude that \(u\) is Hölder continuous at \(x_0\) is exactly as in the proof of \cite[Theorem 8]{Serrin1964} thus we omit it.
\end{proof}

\section{Behavior at infinity}\label{sec-infty}

In this section we obtain a decay estimate for weak solutions to the equation
\begin{equation}\label{extremal-eq}
\left\{
\begin{aligned}
-\di(w\abs{\nabla u}^{p-2}\nabla u)&=w\abs{u}^{q-2}u&&\text{in }\Omega\\
u&\in D^{1,p,w}(\Omega)&&
\end{aligned}
\right.
\end{equation}
where \(q=\chi_wp=\frac{D_wp}{D_w-p}\) and the set \(\Omega\subseteq\RR^N\) (bounded or not) is such that there exists a constant \(S_{p,w}(\Omega)=S_{p,w}>0\) for which the global weighted Sobolev inequality \eqref{sob-ineq-w} holds.

With the aid of the results regarding the equation \(\di\A =\B\) we are able to prove that weak solutions to \eqref{extremal-eq} are locally bounded.
\begin{lemma}\label{ext-bd-local}
Let \(q=\chi_wp\) and \(u\in D^{1,p,w}(\Omega)\) be a weak solution of
\[
-\di(w\abs{\nabla u}^{p-2}\nabla u)=w\abs{u}^{q-2}u\qquad\text{in }\Omega.
\]
Then for every \(R>0\) such that \(B_{4R}(x_0)\subseteq\Omega\) then there exists \(C>0\) depending on \(R>0\) and on \(\norm{u}_{D^{1,p,w}(\Omega)}\) such that
\[
\norm{u}_{L^{\infty}(B_R(x_0))} \leq C.
\]
\end{lemma}

\begin{proof}
Observe that equation \eqref{extremal-eq} can be written in the from \(\di\A=\B\) for \(b=c=e=f=g=0\) and \(d=\abs{u}^{q-p}\). We first use \cref{thm-local-inte} in the version mentioned in \cref{loc-glob-rem2} as from that result we know that if \(d\in L^{\frac{D_w}{p},w}(\Omega)\) then for every \(s\geq 1\) and \(R>0\) the weak solution \(u\) satisfies 
\[
\norm{u}_{L^{s,w}(B_{2R}(x_0))}\leq C_{R,s}\norm{u}_{L^{p,w}(B_{4R}(x_0))},
\]
and \(C_{R,s}\) depends on \(s\) and on \(R^p\pt{\int_{B_{4R}(x_0)}\abs{d}^{\frac{D_w}{p}}w}^{\frac{p}D}\). But because \(u\in D^{1,p,w}(\Omega)\) the Sobolev inequality \eqref{sob-ineq-w} tells us that \(u\in L^{q,w}(\Omega)\), hence \(d\in L^{\frac{D_w}{p},w}(B_{4R}(x_0))\Leftrightarrow q=\chi_wp\). In particular, this shows that \(u\in L^{s,w}(B_{2R}(x_0))\) and as a consequence \(d=\abs{u}^{q-p}\in L^{\frac{D_w}{p-\ve},w}(B_{2R}(x_0))\) for every \(0<\ve<p\) therefore we can use \cref{thm-local-bdd} in the version mentioned in \cref{loc-glob-rem1} to conclude that 
\[
\norm{u}_{L^{\infty}(B_R(x_0))}\leq C,
\]
where \(C\) depends on \(R>0\) and the norm of \(u\) in \(D^{1,p,w}(\Omega)\).
\end{proof}

Now we would like to estimate the decay of the \(L^{q,w}\) norm of weak solutions outside balls of large radii in \(\Omega\).
\begin{lemma}\label{lem-tau}
Suppose \(q=\chi_wp\). If \(u\in D^{1,p,w}(\Omega)\) is a weak solution of \eqref{extremal-eq}, then there exists \(R_0>0\) and \(\tau>0\) such that if \(R\geq R_0\) then
\[
\norm{u}_{L^{q,w}(\Omega\setminus B_{R})}\leq \pt{\frac{R_0}{R}}^{\tau}\norm{u}_{L^{q,w}(\Omega\setminus B_{R_0})}.
\]
Here \(B_R\) denotes an arbitrary ball of radius \(R\).
\end{lemma}
\begin{remark}

Observe that in the case of unbounded \(\Omega\) the above gives an estimate near infinity of the \(L^{q,w}\) norm of \(u\). 

\end{remark}

\begin{proof}
Because \(u\in D^{1,p,w}(\Omega )\) then the function \(\varphi=\eta^p u\) for \(\eta\in W^{1,\infty}(\RR^N)\) is a valid test function in
\[
\int_\Omega \abs{\nabla u}^{p-2}\nabla u\nabla\varphi w = \int_\Omega \abs{u}^{q-2}u\varphi w.
\]
On the one hand, using \(ts\leq C_p\ve^{1-p}t^p+\ve s^{p'}\) for suitable small \(\ve>0\) (depending only on \(p\)), we obtain that
\begin{align*}
\int_\Omega \abs{\nabla u}^{p-2}\nabla u\nabla\varphi w &=\int_\Omega \abs{\nabla u}^{p-2}\nabla u\nabla(\eta^p u) w \\
&=\int_\Omega \abs{\nabla u}^{p-2}\nabla u\cdot\pt{\eta^p \nabla u+p\eta^{p-1}u\nabla\eta} w \\
&=\int_\Omega \abs{\eta\nabla u}^{p} w +p\int_\Omega \eta^{p-1}\abs{\nabla u}^{p-2}\nabla u\cdot u\nabla\eta w \\
&\geq \int_\Omega \abs{\eta\nabla u}^{p} w -p\int_\Omega \pt{\ve\abs{\eta\nabla u}^{p}+C_p\ve^{1-p}\abs{u\nabla\eta}^p} w \\
&\geq \pt{1-p\ve}\int_\Omega \abs{\eta\nabla u}^{p} w -C_p\ve^{1-p}\int_\Omega \abs{u\nabla\eta}^p w\\
&=\frac12\int_\Omega \abs{\eta\nabla u}^{p} w -C_p\int_\Omega \abs{u\nabla\eta}^p w.
\end{align*}

On the other hand, since \(q>p\) we can write
\begin{align*}
\int_\Omega \abs{u}^{q-2}u\varphi w &=\int_\Omega \abs{u}^q\eta^p w\\
&=\int_\Omega \abs{u}^{q-p}\abs{\eta u}^p w\\
&\leq \pt{\int_{\supp \eta} \abs{u}^{q} w}^{1-\frac{p}{q}}\pt{\int_\Omega \abs{\eta u}^{q} w}^{\frac{p}{q}}.
\end{align*}

Hence 
\begin{align*}
\int_\Omega \abs{\nabla(\eta u)}^{p} w &=\int_\Omega \abs{\eta\nabla u+u\nabla \eta}^{p} w\\
&\leq 2^{p-1}\int_\Omega \abs{\eta\nabla u}^{p} w+2^{p-1}\int_\Omega \abs{u\nabla \eta}^{p} w\\
&\leq 2^{p-1}\pt{2\int_\Omega \abs{\nabla u}^{p-2}\nabla u\nabla\varphi w +C_p\int_\Omega \abs{u\nabla\eta}^p w} +2^{p-1}\int_\Omega \abs{u\nabla \eta}^{p} w\\
&\leq C_p\int_\Omega \abs{u\nabla\eta}^p w+2^p\pt{\int_{\supp \eta} \abs{u}^{q} w }^{1-\frac{p}{q}}\pt{\int_\Omega \abs{\eta u}^{q} w}^{\frac{p}{q}},\\
\end{align*}
and the global Sobolev inequality \eqref{sob-ineq-w} tells us that there exists a constant \(S_{p,w}\) such that
\[
S_{p,w}^p\pt{\int_\Omega \abs{\eta u}^{q} w}^{\frac{p}{q}}\leq \int_\Omega \abs{\nabla(\eta u)}^{p} w
\]
therefore we obtain
\begin{equation}\label{p*bd1}
S_{p,w}^p\pt{\int_\Omega \abs{\eta u}^{q} w}^{\frac{p}{q}}\leq C_p\int_\Omega \abs{u\nabla\eta}^p w+2^p\pt{\int_{\supp \eta} \abs{u}^{q} w}^{1-\frac{p}{q}}\pt{\int_\Omega \abs{\eta u}^{q} w}^{\frac{p}{q}}.
\end{equation}

We now choose \(\eta\). First of all, because \(\norm{u}_{q,w}\) is finite for any given \(\ve>0\) we can find \(R_0=R_0(\ve)>0\) such that if \(R\geq R_0\) then 
\[
\int_{\Omega\setminus B_{R}}\abs{u}^{q} w \leq \ve.
\]
With this in mind we choose \(R_0>0\) such that 
\[
2^p\pt{\int_{\Omega\setminus B_{R_0}}\abs{u}^{q} w }^{1-\frac{p}q}\leq \frac{S_{p,w}^p}{2},
\]
and we suppose that \(R\geq R_0\) from now on. We consider \(\eta\in W^{1,\infty}(\RR^N)\) such that \(\eta(x)=0\) for \(x\in B_{R}\), \(\eta(x)=1\) for \(x\notin B_{2R}\), and \(\abs{\nabla\eta}\leq CR^{-1}\). If we use such \(\eta\) in \eqref{p*bd1} we obtain a constant \(C>0\) independent of \(R\) such that
\begin{equation}\label{p*bd2}
\pt{\int_\Omega \abs{\eta u}^{q} w}^{\frac1{q}}\leq C\pt{\int_\Omega \abs{u\nabla\eta}^p w}^{\frac1{p}}.
\end{equation}

Additionally, by the choice of \(\eta\) we can suppose that \(\abs{\nabla\eta}\leq CR^{-1}\) and we obtain
\begin{align}
\int_\Omega \abs{u\nabla\eta}^p w &\leq CR^{-p}\int_{\Omega\cap B_{2R}\setminus B_{R}} \abs{u}^p w \notag\\
&\leq CR^{-p}\pt{w(B_{2R})}^{1-\frac{p}{q}}\pt{\int_{\Omega\cap B_{2R}\setminus B_{R}} \abs{u}^{q} w}^{\frac{p}{q}}\notag\\
&\leq CR^{-p}\pt{w(B_{R_0})\pt{\frac{2R}{R_0}}^{D_w}}^{1-\frac{p}{q}}\pt{\int_{\Omega\cap B_{2R}\setminus B_{R}} \abs{u}^{q} w}^{\frac{p}{q}}\notag\\
&=C\pt{\frac{w(B_{R_0})}{R_0^{D_w}}}^{1-\frac{p}{q}}R^{D_w(1-\frac{p}q)-p}\pt{\int_{\Omega\cap B_{2R}\setminus B_{R}} \abs{u}^{q} w}^{\frac{p}{q}}\notag\\
&\leq CR^{D_w(1-\frac{p}q)-p}\pt{\int_{\Omega\cap B_{2R}\setminus B_{R}} \abs{u}^{q} w}^{\frac{p}{q}} \notag\\
&= C\pt{\int_{\Omega\cap B_{2R}\setminus B_{R}} \abs{u}^{q} w}^{\frac{p}{q}} \label{decay1}
\end{align}
where we have used \eqref{w-ball-estimate} and the fact that \(q=\chi_wp\). From \eqref{p*bd2} and \eqref{decay1} we obtain
\[
\int_\Omega \abs{\eta u}^{q} w \leq C\int_{\Omega\cap B_{2R}\setminus B_{R}} \abs{u}^{q} w,
\]
for some constant \(C>0\) depending on \(p,q, R_0\) but independent of \(R\). To continue, observe that since \(\eta\equiv 1\) on \(B_{2R}^c\) we can write
\begin{align*}
\int_{\Omega\setminus B_{2R}}\abs{u}^{q} w &\leq \int_{\Omega}\abs{\eta u}^{q} w\\
&\leq C\int_{\Omega\cap B_{2R}\setminus B_{R}} \abs{u}^{q}w\\
&= C\int_{\Omega\setminus B_{R}} \abs{u}^{q}w -C\int_{\Omega\setminus B_{2R}} \abs{u}^{q} w,
\end{align*}
thus, if \(\theta=\frac{C}{C+1}\in(0,1)\) then we obtain
\begin{equation}\label{bdf}
\int_{\Omega\setminus B_{2R}}\abs{u}^{q} w \leq \theta\int_{\Omega\setminus B_{R}} \abs{u}^{q} w.
\end{equation}

Consider now \(f(R)=\int_{\Omega\setminus B_{R}}\abs{u}^{q} w\), then \eqref{bdf} tells us that there exists \(\theta\in(0,1)\) such that for every \(R\geq R_0\) one has
\[
f(2R)\leq \theta f(R),
\]
in particular one could take \(R=R_n=2^nR_0\) for \(n\geq 0\) and conclude that \(f(2^{n}R_0)\leq \theta f(2^{n-1}R_0)\), or after iterating
\[
f(2^{n}R_0)\leq \theta^nf(R_0).
\]
Observe that if \(R\geq R_0\) then one can find \(n\geq 1\) such that \(2^{n-1}R_0\leq R<2^{n}R_0\) then \(n> \log_2(RR_0^{-1})\) and as a consequence we obtain 
\[
f(R)\leq f(2^{n}R_0)\leq \theta^n f(R_0)\leq \theta^{\log_2(RR_0^{-1})}f(R_0).
\]
and because \(x^{\log_2 y}=y^{\log_2x}\) so we have shown
\[
\int_{\Omega\setminus B_{R}}\abs{u}^{q} w\leq \pt{\frac{R_0}{R}}^{-\log_2\theta}\int_{\Omega\setminus B_{R_0}}\abs{u}^{q} w
\]
and the result is proved for \(\tau:=-\frac1q\log_2\theta>0\).
\end{proof}

\begin{lemma}\label{lema-tau2}
Suppose that \(q=\chi_wp\) and that \(u\in D^{1,p,w}(\Omega)\) is a weak solution of
\[
-\di(w\abs{\nabla u}^{p-2}\nabla u)=w\abs{u}^{q-2}u\qquad\text{in }\Omega.
\]
Then for each \(s>q\) there exists \(R_0>0\) (depending on \(s\)) such that if \(R\geq R_0\) then there exists \(C=C(p,q,w;s)>0\) for which 
\[
\norm{u}_{L^{s,w}(\Omega\setminus B_{2R})}\leq \frac{C}{R^{\frac{D_w-p}{p}-o_s(1)}}\norm{u}_{L^{q,w}(\Omega\setminus B_{R})},
\]
where \(o_s(1)\) is a quantity that goes to \(0\) as \(s\to \infty\).
\end{lemma}

\begin{proof}
Firstly notice that thanks to the \(L^{s,w}\) interpolation inequality it is enough to exhibit a sequence \(s_n\sublim\limits_{n\to\infty}+\infty\) for which one has
\[
\norm{u}_{L^{s_n,w}(\Omega\setminus B_{2R})}\leq \frac{C}{R^{\frac{p}{q-p}-o_n(1)}}\norm{u}_{L^{q,w}(\Omega\setminus B_{R})}.
\]
Then we observe that in the context of \eqref{model-eq} we can view the equation as
\[
\di\A=\B
\]
where \(b=c=e=f=g=0\) and \(d=\abs{u}^{q-p}\). The assumption \(u\in D^{1,p,w}(\Omega)\) tells us that \(\varphi=\eta^pG(u)\) is a valid test function and we can follow the notation of the proof \cref{thm-local-bdd}, in fact, since \(e=f=g=0\) we can further suppose that \(k>0\) is arbitrary in the definition of both \(F\) and \(G\). Starting with \eqref{bas-esti0} we now integrate over \(\Omega\) to obtain
\[
\int_{\Omega}\abs{\eta\nabla v}^pw\leq p\int_{\Omega}\abs{v\nabla \eta}\abs{\eta\nabla v}^{p-1}w+(\beta+1)\alpha^{p-1}\int_{\Omega}d\abs{v\eta}^pw,
\]
where \(v=F(\bar u)\) and \(\beta=1+p(\alpha-1)\). From the above and the global Sobolev inequality \eqref{sob-ineq-w} we obtain
\[
\pt{\int_\Omega \abs{\eta v}^qw}^{\frac{p}q}\leq C_\alpha\pt{\int_{\Omega}\abs{v\nabla \eta}^pw+\int_{\Omega}d\abs{v\eta}^pw},
\]
for some \(C_\alpha=O(\alpha^p)\). We can pass to the limit \(l\to\infty\) the above inequality to deduce that
\[
\pt{\int_\Omega \abs{\eta \bar u^\alpha}^qw}^{\frac{p}q}\leq C_\alpha\pt{\int_{\Omega}\abs{\nabla \eta }^p\abs{\bar u}^{\alpha p}w+\int_{\Omega}\eta^p\abs{u}^{q-p}\abs{\bar u}^{p\alpha}w},
\]
where we have used \(d=\abs{u}^{q-p}\). If we pass to the limit \(k\to 0^+\) then we reach 
\[
\pt{\int_\Omega \abs{\eta u^\alpha}^qw}^{\frac{p}q}\leq C_\alpha\pt{\int_{\Omega}\abs{\nabla \eta }^p\abs{u}^{\alpha p}w+\int_{\Omega}\eta^p\abs{u}^{q+p(\alpha-1)}w}.
\]

Observe that because \(q>p\) we can do the following estimate 
\begin{align*}
\int_{\Omega}\eta^pu^{q+p(\alpha-1)} w &=\int_\Omega u^{q-p}\pt{\eta u^{\alpha}}^p w\\
&\leq \pt{\int_{\supp \eta} \abs{u}^{q} w}^{1-\frac{p}{q}}\pt{\int_\Omega \abs{\eta u^{\alpha}}^{q} w}^{\frac{p}{q}},
\end{align*}
therefore we have
\[
\pt{\int_\Omega \abs{\eta u^\alpha}^qw}^{\frac{p}q}\leq C_\alpha\int_{\Omega}\abs{\nabla \eta }^p\abs{u}^{\alpha p}w+C_\alpha\pt{\int_{\supp \eta} \abs{u}^{q} w}^{1-\frac{p}{q}}\pt{\int_\Omega \abs{\eta u^{\alpha}}^{q} w}^{\frac{p}{q}}.
\]

We now select \(\eta\). Because \(u\in D^{1,p,w}(\Omega)\) and that \eqref{sob-ineq-w} holds then we know that \(u\in L^{q,w}(\Omega)\), therefore for any given \(\nu>0\) we can find \(R_0=R_0(\nu)>0\) such that
\[
\int_{\Omega\cap\set{\abs{x}\geq R}}\abs{u}^{q} w \leq \nu,\qquad\forall\, R\geq R_0.
\]
With this in mind we choose \(R_0=R_0(\alpha)>0\) such that 
\[
C_\alpha\pt{\int_{\Omega\cap\set{\abs{x}\geq R_0}} \abs{u}^{q}w}^{1-\frac{p}q}\leq \frac{1}{2},
\]
and we suppose that \(R\geq R_0\).

For \(n\geq 0\) we consider \(R_n=R(2-2^{-n})\) and a smooth non-negative function \(\eta\) such that \(\eta(x)=0\) for \(\abs{x}\leq R_{n}\), \(\eta(x)=1\) for \(\abs{x}\geq R_{n+1}\) and satisfies \(\abs{\nabla\eta}\leq \frac{C2^n}{R}\), 
\begin{gather*}
\supp\eta\subseteq \Omega\setminus B_{R_n}\\
\supp\nabla\eta \subseteq \Omega\cap B_{R_{n+1}}\setminus B_{R_{n}},
\end{gather*}
therefore
 so that
\[
\pt{\int_\Omega \abs{\eta u^\alpha}^qw}^{\frac{1}{\alpha q}}\leq 2^{\frac1{\alpha p}}C_\alpha^{\frac{1}{p\alpha}}\pt{\int_{\Omega}\abs{\nabla \eta }^p\abs{u}^{\alpha p}w}^{\frac1{p\alpha}}.
\]
and if for \(n\geq 1\) we take \(\alpha_n=\pt{\frac{q}{p}}^{n}\) then we obtain
\[
\pt{\int_{\Omega\setminus B_{R_{n+1}}}\abs{u}^{\frac{q^{n+1}}{p^{n}}} w}^{\frac{p^{n}}{q^{n+1}}}\leq 2^{\frac{p^{n-1}}{q^{n}}}C_{\alpha_n}^{\frac{p^{n-1}}{q^{n}}} R^{-\frac{p^{n}}{q^{n}}}\pt{\int_{\Omega\setminus B_{R_n}}\abs{u}^{\frac{q^{n}}{p^{n-1}}} w}^{\frac{p^{n-1}}{q^{n}}},
\]
or equivalently, if \(s_n=\frac{q^{n}}{p^{n-1}}\) and \(\mathcal U_n=\norm{u}_{L^{s_n,w}(\Omega\setminus B_{R_{n}})}\), 
\[
\mathcal U_{n+1}\leq \frac{2^{\frac{p^{n-1}}{q^{n}}}C_{\alpha_n}^{\frac{p^{n-1}}{q^{n}}}}{R^{\frac{p^{n}}{q^n}}}\mathcal U_n,
\]
which after iterating gives
\[
\mathcal U_n\leq \pt{\frac{\prod_{i=1}^{n-1} 2^{\frac{p^{i-1}}{q^{i}}}C_{\alpha_i}^{\frac{p^{i-1}}{q^{i}}}}{R^{\sum_{i=1}^{n-1}\pt{\frac{p}{q}}^i}}}\mathcal U_1.
\]

If we observe that \(C_\alpha=O(\alpha^p)\) so that for every large \(i\) we have 
\[
C_{\alpha_i}^{\frac{p^{i-1}}{q^{i}}}\leq C^{\frac{p^{i-1}}{q^{i}}} \pt{\frac{q}{p}}^{i\frac{p^{i}}{q^{i}}}
\]
for some constant \(C>0\), and thus the product \(\prod_{i=1}^{n-1} 2^{\frac{p^{i-1}}{q^{i}}}C_{\alpha_i}^{\frac{p^{i-1}}{q^{i}}}\) is convergent because \(q>p\), then the result follows by noticing that \(\mathcal U_1\leq \norm{u}_{L^{q,w}(\Omega\setminus B_R)}\), \(\mathcal U_n\geq \norm{u}_{L^{s_n,w}(B_{2R})}\), and that
\[
\sum_{i=1}^{n-1}\pt{\frac{p}{q}}^i
=\frac{p}{q-p}-\frac{q}{q-p}\pt{\frac{p}{q}}^{n}=\frac{p}{q-p}-o_n(1),
\]
because \(q=\frac{D_wp}{D_w-p}>p\).
\end{proof}

Now we are in position to prove \cref{decay-thm}:

\begin{proof}[Proof of \cref{decay-thm}]
Consider the value of \(R_0>0\) given in \cref{lem-tau}, and suppose that \(x\in \Omega\setminus B_{2R_0}\). Fix \(0<r<\frac{R_0}{4}\) so that \(B_{2r}(x)\subseteq\Omega\) and use \cref{ext-bd-local} to obtain
\[
\abs{u(x)}\leq \norm{u}_{L^{\infty}(B_r(x))}\leq C_r\norm{u}_{L^{p}(B_r(x))}\leq C_{r,s}\norm{u}_{L^{s}(B_r(x))},
\]
for any \(s>p\). If we consider \(R=\frac{\abs{x}}4\), then by geometric considerations we deduce that \(B_{2r}(x)\subseteq \Omega\setminus B_{2R}\) hence
\[
\norm{u}_{L^{s,w}(B_{2r}(x))}\leq \norm{u}_{L^{s,w}(\Omega\setminus B_{2R})}.
\]

Now we select \(s>q\) large enough so that \(o_s(1)\leq \frac{\tau}{2}\) in \cref{lema-tau2}, where \(\tau>0\) is taken from \cref{lem-tau}, by doing that we obtain
\begin{align*}
\norm{u}_{L^{s,w}(\Omega\setminus B_{2R})}&\leq \frac{C}{R^{\frac{p}{q-p}-o_s(1)}}\norm{u}_{L^{q,w}(\Omega\setminus B_{R})}\\
&\leq \frac{C}{R^{\frac{p}{q-p}-\frac{\tau}2}}\norm{u}_{L^{q,w}(\Omega\setminus B_{R})}\\
&\leq \frac{C}{R^{\frac{p}{q-p}-\frac{\tau}{2}}}\pt{\frac{R_0}{R}}^{\tau}\norm{u}_{L^{q,w}(\Omega\setminus B_{R_0})},
\end{align*}
therefore, by putting all together we obtain
\[
\abs{u(x)}\leq \frac{CR_0^{\tau}}{R^{\frac{p}{q-p}+\frac{\tau}{2}}}\norm{u}_{L^{q,w}(\Omega\setminus B_{R_0})}=\frac{C}{\abs{x}^{\frac{p}{q-p}+\lambda}},
\]
for some constant \(C>0\) independent of \(\abs{x}\geq 2R_0\), and the result is proved for \(\tilde R=2R_0\).
\end{proof}

\section{A better decay estimate for some particular weights \(w\)}\label{examples-sect}

The result from \cref{decay-thm} can be improved if one knows two facts regarding the weighted \(p\)-Laplace operator \(L_w(u)=-\di(w\abs{\nabla u}^{p-2}\nabla u)\):
\begin{itemize}
\item The operator satisfies a comparison principle, meaning that if \(L_w(u)\leq L_w(v)\) weakly in \(\Omega\) and \(u\leq v\) over \(\domega\) then \(u\leq v\) in \(\Omega\).
\item There exists a constant \(C>0\) such that
\[
L_w(\abs{x}^{s_1})\geq Cw\abs{x}^{s_2}
\]
for suitable \(s_1,s_2\) and large \(\abs{x}\).
\end{itemize}

The first condition is easily verified in general (and in a stronger version for particular weights) as one can see in \cref{sec:comparison}, however the second condition depends heavily on the type of weight. In this section we discuss some particular cases that have been of interest recently.

\subsection{The unweighted case \(w=1\)}

As we mentioned in the introduction, the decay estimate we have obtained is an adaptation of a result from \cite{CPY2012} in the unweighted case, namely in that work they prove

\begin{proposition}[Lemma B.3 in \cite{CPY2012}]
If \(u\in W^{1,p}(\RR^N)\) is a weak solution of
\[
-\Delta_p u=\abs{u}^{p^*-2}u
\]
then there exists \(\lambda>0\) such that
\[
\abs{u(x)}\leq \frac{C}{1+\abs{x}^{\frac{N-p}{p}+\lambda}}.
\]
\end{proposition}
The above result is improved by using the comparison principle for \(\Delta_p\) to obtain

\begin{proposition}[Proposition B.1 in \cite{CPY2012}]
If \(u\in W^{1,p}(\RR^N)\) is a weak solution of
\[
-\Delta_p u=\abs{u}^{p^*-2}u
\]
then
\[
\abs{u(x)}\leq \frac{C}{1+\abs{x}^{\frac{N-p}{p-1}-\ve}}
\]
for every \(\ve>0\).
\end{proposition}

\subsection{Monomial weights: \(w=x^A\)}

In the case of monomial weights
\[
w_A(x)=x^A=\abs{x_1}^{a_1}\cdot\ldots\cdot \abs{x_N}^{a_N}
\]
we can also improve the decay estimate from \cref{decay-thm}. It is clear that \(w_A\) verifies \eqref{inv-loc-int2} and that \(w_i(x):=\abs{x_i}^{a_i}\) belongs to \(L^1_{loc}(\RR)\) for \(a_i>-1\). Also  \(w_i\) is doubling (see for instance \cite[Corollary 15.35]{HeKiMa2006}), moreover, if \(a_i\geq 0\) one can obtain the best doubling constant \(\gamma_w\) as follows: observe that for \(B_R(x_0)=(x_0-R,x_0+R)\) we have
\[
w_i(B_R(x_0))=R^{1+a_i}w_i(B_1(\frac{x_0}R)),
\]
therefore
\[
\frac{w_i(B_{2R}(x_0))}{w_i(B_R(x_0))}=2^{1+a_i}\frac{w_i(B_{1}(\frac{x_0}{2R}))}{w_i(B_1(\frac{x_0}R))}\leq 2^{1+a_i}
\]
because of the following
\begin{lemma}\label{doub.const-pow}
For any \(x_0\in \RR\) and \(a>0\) we have
\[
w_a(B_1(x_0))\leq w_a(B_1(2x_0)),
\]
with equality if \(x_0=0\).
\end{lemma}
We will prove this lemma in \cref{sec:doubling}, but observe that as a consequence we deduce that for the cube\footnote{So far we have worked with balls, but the geometry of \(\RR^N\) allows us to work with cubes instead by only including a geometric constant when needed.} \(B_R=\prod_{i=1}^N(x_{0,i}-R,x_{0,i}+R)\) one has
\[
\frac{w_A(B_{2R}(x_0))}{w_A(B_R(x_0))}\leq 2^{N_A}
\]
for \(N_A=N+a_1+a_2+\ldots+a_N\) and that \(N_A\) is the smallest possible choice for the exponent on the right-hand side. This shows that \(D_A=N_A\), which agrees with \cite{CR-O2013-2,Cas2016-2} where it was proved that for every \(A=(a_1,\ldots,a_N)\in\RR^N\) with \(a_i\geq 0\) there exists a constant \(S_{p,A}>0\) such that
\begin{equation}\label{sob-ineq-mono}
S_{p,A}\pt{\int_{\RR^N_A}\abs{u}^qx^A\dx}^{\frac1q}\leq \pt{\int_{\RR^N_A}\abs{\nabla u}^px^A\dx}^{\frac1p},\qquad\forall\, u\in C^\infty_c(\RR^N_A)
\end{equation}
where \(q=\frac{N_Ap}{N_A-p}\) and \(\RR^N_A\) is defined as
\[
\RR^N_A=\set{(x_1,\ldots,x_N)\in \RR^N: x_i>0 \text{ whenever } a_i>0},
\]
therefore the Sobolev exponent \(q\) from \eqref{sob-ineq-mono} coincides with the local Sobolev exponent obtained in \cref{local-sobolev-remark} from the doubling constant \(\gamma_A=2^{N_A}\).

Observe also that \cite[Section 5]{CasCo2022} tells us that the weight \(w_A\) verifies
\[
\int_{B}\abs{u-u_{B,x^A}}^px^A\D x\leq Cl(B)^p\int_{B}\abs{\nabla u}^px^A\D x \qquad \text{ for all balls }B\subset \RR^N_A,
\]
so that \ref{weighted-poincare} is readily satisfied for \(w_A\).

If \(H^{1,p,A}_0(\RR^N_A)\) denotes the closure of \(C^\infty_c(\RR^N_A)\) then the results of \cref{sec-infty} are valid, and with the aid of the comparison principle \cref{comparison-w} we obtain the following

\begin{theorem}\label{decay-mono}
Suppose \(A=(a_1,a_2,\ldots,a_N)\in\RR^N\), \(N_A>p\) with \(a_i\geq 0\) for all \(i\). If \(u\in H^{1,p,A}_0(\RR^N_A)\) is a weak solution of
\[
\left\{
\begin{aligned}
-\di(x^{A}\abs{\nabla u}^{p-2}\nabla u)&=x^A\abs{u}^{q-2}u&&\text{in }\RR^N_A,\\
u&=0&&\text{on }\partial(\RR^N_A),
\end{aligned}
\right.
\]
where \(q=\frac{N_Ap}{N_A-p}\). Then
\[
\abs{u(x)}\leq \frac{C}{1+\abs{x}^{\frac{N_A-p}{p-1}-\ve}}
\]
for every \(\ve>0\).
\end{theorem}

\begin{proof}
From \cref{decay-thm} we know that there exists \(\lambda>0\) and constants \(R_0\geq 1\), \(C_1>0\) such that
\[
\abs{u(x)}\leq C_1\abs{x}^{-t}\quad \forall\, \abs{x}\geq R_0,
\]
where \(t=\frac{N_A}{p}-1+\lambda\). We notice that there exists \(\sigma>0\) such that
\[
t+2+\sigma+(p-2)(t+\sigma+1)<t(q-1),
\]
and we observe that for any \(1<s<\frac{N_A-p}{p-1}\) we have
\begin{equation*}
-\di\pt{x^A\abs{\nabla\pt{\abs{x}^{-s}}}^{p-2}\nabla\pt{\abs{x}^{-s}}}=C_2x^A\abs{x}^{-s-2-(p-2)(s+1)},
\end{equation*}
where \(C_2=s\abs{s}^{p-2}(N_A-(p-2)(s+1)-2-s)>0\). Hence for \(\abs{x}\geq R_0\) and \(u_+=\max\set{u,0}\) we have
\begin{align*}
-\di\pt{x^A\abs{\nabla u_+}^{p-2}\nabla u_+}&\leq x^A u_+^{q-1}\\
&\leq C_1x^A\abs{x}^{-t(q-1)}\\
&\leq C_1x^A\abs{x}^{-(t+\sigma)-2-(p-2)((t+\sigma)+1)}\\
&= -\frac{C_1}{C_2}\di\pt{x^A\abs{\nabla\pt{\abs{x}^{-(t+\sigma)}}}^{p-2}\nabla\pt{\abs{x}^{-(t+\sigma)}}}.
\end{align*}
Because \(u_+=0\) over \(\partial(\RR^N_A)\) the comparison principle tells us that
\[
u_+(x)\leq C_3\abs{x}^{-t-\sigma}\quad\forall\, \abs{x}\geq R_0,
\]
for some sufficiently large constant \(C_3>0\). We can iterate this process for \(\tilde t=t+\sigma\) provided we can find \(\tilde \sigma>0\) such that \(\tilde t+2+\tilde \sigma+(p-2)(\tilde t+\tilde\sigma+1)<\frac{N_A-p}{p-1}\), hence we deduce that 
\[
u_+(x)\leq C\abs{x}^{-\frac{N_A-p}{p-1}+\ve}\quad\forall\, \abs{x}\geq R_0.
\]
for every \(\ve>0\). The above can be repeated for \((-u)_+\) and the result is proved.
\end{proof}

If we use \cref{comparison-mono} instead of \cref{comparison-w} we obtain the following variant of \cref{decay-mono}. If we observe that 
\begin{align*}
\partial(\RR^N_A)&=\bigcup_{i=1}^N\set{x\in\RR^N: x_i=0}\\
&=\bigcup_{i\in I_1}\set{x\in\RR^N: x_i=0}\cup \bigcup_{i\in I_2}\set{x\in\RR^N: x_i=0}\\
&=:\Gamma_1 \cup \Gamma_2,
\end{align*}
where \(I_1=\set{i: 0\leq a_i<1}\) and \(I_2=\set{i: a_i\geq 1}\) then we have

\begin{theorem}\label{decay-mono2}
Suppose \(A=(a_1,a_2,\ldots,a_N)\in\RR^N\), \(N_A>p\) with \(a_i\geq 0\) for all \(i\). If \(u\in H^{1,p,A}(\RR^N_A)\) is a weak solution of
\[
\left\{
\begin{aligned}
-\di(x^{A}\abs{\nabla u}^{p-2}\nabla u)&=x^A\abs{u}^{q-2}u&&\text{in }\RR^N_A,\\
u&=0&&\text{on }\Gamma_1,
\end{aligned}
\right.
\]
where \(q=\frac{N_Ap}{N_A-p}\). Then
\[
\abs{u(x)}\leq \frac{C}{1+\abs{x}^{\frac{N_A-p}{p-1}-\ve}}
\]
for every \(\ve>0\).
\end{theorem}

\begin{proof}
The proof goes exactly as the previous proof, the only difference being that \cref{comparison-mono} tells us that we do not need \(u_+(x)\leq C_3\abs{x}^{-t-\sigma}\) over the set \(\Gamma_2\) to obtain said inequality all over the set \(\set{x\in \RR^N_A: \abs{x}\geq R_0}\). We omit the details.
\end{proof}

\subsection{Power type weights: \(w=\abs{x}^a\)}
In the case of the power type weights \(w(x)=\abs{x}^a\) we can also improve the decay estimate from \cref{decay-thm}. First of all we recall that the result of Caffarelli-Kohn-Nirenberg \cite{CKN1984} which tells us that for every \(a\in \RR\) there exists a constant \(S_{p,a}>0\) such that
\begin{equation}\label{sob-ineq-power}
S_{p,a}\pt{\int_{\RR^N_a}\abs{u}^q\abs{x}^a\dx}^{\frac1q}\leq \pt{\int_{\RR^N_a}\abs{\nabla u}^p\abs{x}^a\dx}^{\frac1p},
\end{equation}
where \(q=\frac{N_ap}{N_a-p}\), \(N_a=N+a\) and \(\RR^N_0\) is the punctured plane \(\RR^N_0=\RR^N\setminus\set{0}\). Moreover, the weight \(w_a(x)=\abs{x}^a\) satisfies \(w_a(B_R(x_0))=R^{N+a}w_a(B_1(\frac{x_0}{R}))\) so that
\[
\frac{w_a(B_{2R}(x_0))}{w_a(B_R(x_0))}=2^{N+a}\frac{w_a(B_1(\frac{x_0}{2R}))}{w_a(B_1(\frac{x_0}{R}))}\leq 2^{N+a}
\]
because we have
\begin{lemma}\label{doub.const-pow2}
For any \(x_0\in \RR^N\) and \(a>0\) we have
\[
w_a(B_1(x_0))\leq w_a(B_1(2x_0)),
\]
with equality if \(x_0=0\).
\end{lemma}

We prove this lemma below in \cref{sec:doubling}, and because the operator \(-\di(\abs{x}^a\abs{\nabla u}^{p-2}\nabla u)\) also satisfies the comparison principle \cref{comparison-mono} we have
\begin{theorem}
Suppose \(a>0\) and \(N_a>p\). If \(u\in H^{1,p,a}(\RR^N)\) is a weak solution of
\[
-\di(\abs{x}^{a}\abs{\nabla u}^{p-2}\nabla u)=\abs{x}^a\abs{u}^{q-2}u,
\]
where \(q=\frac{N_ap}{N_a-p}\). Then
\[
\abs{u(x)}\leq \frac{C}{1+\abs{x}^{\frac{N_a-p}{p-1}-\ve}}
\]
for every \(\ve>0\).
\end{theorem}

\begin{proof}
This is analogous to the proof of \cref{decay-mono}. We again deduce the existence of \(\lambda>0\) and constants \(R_0\geq 1\), \(C_1>0\) such that
\[
\abs{u(x)}\leq C_1\abs{x}^{-t}\quad \forall, \abs{x}\geq R_0,
\]
where \(t=\frac{N_a}{p}-1+\lambda\). Now for any \(1<s<\frac{N_a-p}{p-1}\) we have
\begin{align*}
-\di\pt{\abs{x}^a\nabla\pt{\abs{x}^{-s}}}&=s\abs{s}^{p-2}(N_a-(p-2)(s+1)-2-s)\abs{x}^{a-(p-2)(s+1)-s-2}\\
&=C_2\abs{x}^{a-(p-2)(s+1)-s-2},
\end{align*}
where \(C_2>0\). Hence for \(\abs{x}\geq R_0\) and \(u_+=\max\set{u,0}\) we have
\[
-\di\pt{\abs{x}^a\abs{\nabla u_+}^{p-2}\nabla u_+}\leq -\frac{C_1}{C_2}\di\pt{\abs{x}^a\abs{\nabla\pt{\abs{x}^{-(t+\sigma)}}}^{p-2}\nabla\pt{\abs{x}^{-(t+\sigma)}}},
\]
where \(\sigma>0\) verifies 
\[
t+2+\sigma+(p-2)(t+\sigma+1)<t(q-1).
\]
We use the comparison principle \cref{comparison-mono} and the same iterative argument from the proof \cref{decay-mono} to obtain a constant \(C>0\) for which
\[
u_+(x)\leq C\abs{x}^{-\frac{N_a-p}{p-1}+\ve}\quad\forall\, \abs{x}\geq R_0.
\]
holds for every \(\ve>0\). Finally, because the same estimate can be obtained for \((-u)_+\) the proof is completed.
\end{proof}

\appendix

\section{Doubling constant of power type weights}\label{sec:doubling}

Observe that \cref{doub.const-pow} is just \cref{doub.const-pow2} for the case \(N=1\), so we only need to prove the latter.
\begin{proof}[Proof of \cref{doub.const-pow2}]
We separate the proof into two cases: \(\abs{x_0}\geq2\) and \(0\leq \abs{x_0}<2\). Observe that
\begin{gather*}
x\in B_1(x_0)\Rightarrow \abs{x}< 1+\abs{x_0},\\
x\in B_1(2x_0)\Rightarrow \abs{x}> 2\abs{x_0}-1,
\end{gather*}
so \(1+\abs{x_0}\leq 2\abs{x_0}-1\Leftrightarrow \abs{x_0}\geq 2\) then \(B_1(x_0)\) and \(B_1(2x_0)\) are disjoint and we can write
\begin{align*}
w_a(B_1(x_0))&=\int_{B_1(x_0)}\abs{x}^a\dx\\
&\leq (1+\abs{x_0})^a\lambda(B_1(x_0))\\
&\leq (2\abs{x_0}-1)^a\lambda(B_1(2x_0))\\
&\leq \int_{B_1(2x_0)}\abs{x}^a\dx\\
&=w_a(B_1(2x_0)),
\end{align*}
due to the invariance of the Lebesgue measure with respect to translations. If \(0\leq \abs{x_0}<2\) then \(E=B_1(x_0)\cap B_1(2x_0)\neq \varnothing\) and we have
\begin{align*}
w_a(B_1(x_0))&=w_a(B_1(x_0)\setminus B_1(2x_0))+w_a(E)\\
&=w_a(B_1(2x_0))+w_a(B_1(x_0)\setminus B_1(2x_0))-w_a(B_1(2x_0)\setminus B_1(x_0)),
\end{align*}
and if \(x\in B_1(x_0)\setminus B_1(2x_0)\) then \(\abs{x-x_0}<1\) with \(\abs{x-2x_0}\geq 1\) therefore
\[
\abs{x}^2<1+2\abs{x_0}^2.
\]
Similarly, if \(x\in B_1(2x_0)\setminus B_1(x_0)\) then \(\abs{x-2x_0}<1\) with \(\abs{x-x_0}\geq 1\) so that
\[
\abs{x}^2>1+2\abs{x_0}^2,
\]
and by the invariance of the Lebesgue measure under reflections we see that
\begin{align*}
w_a(B_1(x_0)\setminus B_1(2x_0))&\leq (1+2\abs{x_0}^2)^{\frac{a}2}\lambda(B_1(x_0)\setminus B_1(2x_0))\\
&=(1+2\abs{x_0}^2)^{\frac{a}2}\lambda(B_1(2x_0)\setminus B_1(x_0))\\
&\leq w_a(B_1(2x_0)\setminus B_1(x_0)),
\end{align*}
so that
\[
w_a(B_1(x_0))\leq w_a(B_1(2x_0)),
\]
and the proof is completed.
\end{proof}

\section{Comparison principles}\label{sec:comparison}

A simple comparison principle can be easily obtained for the operator \(L_w(u)=-\di(w\abs{\nabla u}^{p-2}\nabla u)\) as it can be seen in the following
\begin{proposition}\label{comparison-w}
For a connected open set \(\Omega\) suppose \(u,v\in D^{1,p,w}(\Omega)\) satisfy \(u\leq v\) over \(\domega\) and 
\[
-\di\pt{w\abs{\nabla u}^{p-2}\nabla u}\leq -\di\pt{w\abs{\nabla v}^{p-2}\nabla v},
\]
weakly, that is
\[
\int_{\Omega}\pt{\abs{\nabla u}^{p-2}\nabla u-\abs{\nabla v}^{p-2}\nabla v}\nabla \varphi w\dx\leq 0\quad\forall\, \varphi\in H^{1,p,w}_0(\Omega),\ \varphi\geq 0,
\]
then \(u\leq v\) in \(\Omega\).
\end{proposition}

\begin{proof}
The function \(\varphi=(u-v)_+\) is a valid test function hence we obtain
\[
\int_{\Omega^+} \pt{\abs{\nabla u}^{p-2}\nabla u-\abs{\nabla v}^{p-2}\nabla v}\cdot\pt{\nabla u-\nabla v}w\leq 0,
\]
where \(\Omega^+=\set{x\in\Omega : u(x)>v(x)}\). However the integrand is non-negative due to the convexity of the function \(s\mapsto \abs{s}^p\), so there are two possibilities, either \(\Omega^+=\varnothing\) in which case we obtain that \(u\leq v\) and the result is proven, or \(\nabla u=\nabla v\) in which case \(u=v+C\) on \(\Omega^+\) for some constant \(C\). But since \(u=v\) on \(\partial\Omega^+\) we conclude that \(u=v\) on \(\Omega^+\) which is impossible.
\end{proof}

This proposition is enough for the proof of \cref{decay-mono} but it is not enough for \cref{decay-mono2}. A standard comparison principle for an operator \(L\) could be: If \(Lu\leq Lv\) in \(\Omega\) and \(u\leq v\) over \(\domega\) then \(u\leq v\) in \(\Omega\), however if the operator \(Lu\) is \(-\di\pt{x^A\abs{\nabla u}^{p-2}\nabla u}\) then it is only needed to impose the condition \(u\leq v\) over a \emph{portion of the boundary}. Recall that the monomial weight is defined as
\[
x^A=\prod_{i=1}^N\abs{x_i}^{a_i},
\]
where \(a_i\geq 0\) for all \(i\in\set{1,2,\ldots,N}\) and that the study made in \cite{Cas2016-2,Cas2021} pointed out that the cases where \(i\) is such that \(a_i\geq 1\) might require a different treatment, therefore we consider the following partition of the set \(\set{1,2,\ldots,N}\):
\[
I_1=\set{i\in \set{1,2,\ldots,N} : 0\leq a_i<1},\quad I_2=\set{i\in \set{1,2,\ldots,N} : a_i\geq 1},
\]
and define
\[
\Gamma_1=\set{x\in \RR^N: x_i=0\text{ for some }i\in I_1},\quad \Gamma_2=\set{x\in \RR^N: x_i=0\text{ for some }i\in I_2}.
\]
The following result says that in order to have a comparison principle it is enough to impose the inequality \(u\leq v\) over the portion of \(\domega\) not intersecting \(\Gamma_2\).

\begin{proposition}\label{comparison-mono}
For a connected open set \(\Omega\) suppose \(u,v\in D^{1,p,A}(\Omega)\) satisfy \(u\leq v\) over \(\domega\setminus \Gamma_2\) and 
\[
-\di\pt{x^A\abs{\nabla u}^{p-2}\nabla u}\leq -\di\pt{x^{A}\abs{\nabla v}^{p-2}\nabla v},
\]
weakly then \(u\leq v\) in \(\Omega\).
\end{proposition}

\begin{proof}
Consider \(\eta_\ve\in C^\infty(\RR^N)\) defined as \(\eta_\ve(x)=\prod_{i\in I_2}\rho_\ve(x_i)\) where \(\rho_\ve\in C^\infty(\RR)\) satisfies
\[
\rho(x)=0\text{ if }x\leq \ve,\quad \rho(x)=1\text{ if }x\geq 2\ve,\quad \abs{\rho'(x)}\leq C\ve^{-1}.
\]

Observe that \(\varphi=\eta_\ve (u-v)^+\) is a valid test function because \(u-v\leq 0\) over \(\domega\setminus \Gamma_2\) and \(\eta_\ve=0\) over \(\Gamma_2\). Hence one obtains
\begin{multline*}
\int_{\Omega^+} \pt{\abs{\nabla u}^{p-2}\nabla u-\abs{\nabla v}^{p-2}\nabla v}\cdot\pt{\nabla u-\nabla v}\eta_\ve x^A\dx\\
+\int_{\Omega^+} \pt{\abs{\nabla u}^{p-2}\nabla u-\abs{\nabla v}^{p-2}\nabla v}\cdot\nabla\eta_\ve\pt{u-v}x^A\dx\leq 0
\end{multline*}
where \(\Omega^+=\set{x\in\Omega : u(x)>v(x)}\). As a consequence we get
\begin{multline*}
\int_{\Omega^+} \pt{\abs{\nabla u}^{p-2}\nabla u-\abs{\nabla v}^{p-2}\nabla v}\cdot\pt{\nabla u-\nabla v}\eta_\ve x^A\dx\\
\leq \abs{\int_{\Omega^+} \pt{\abs{\nabla u}^{p-2}\nabla u-\abs{\nabla v}^{p-2}\nabla v}\cdot\nabla\eta_\ve\pt{u-v} x^A\dx}.
\end{multline*}
However
\begin{align*}
\abs{\int_{\Omega^+} u\abs{\nabla u}^{p-2}\nabla u\cdot\nabla\eta_\ve x^A\dx}&\leq \pt{\int_{\Omega^+} \abs{u\nabla\eta_\ve}^px^A\dx}^{\frac1p}\pt{\int_{\supp \nabla\eta_\ve} \abs{\nabla u}^{p}x^A\dx}^{1-\frac1p}\\
&\leq C\pt{\int_{\supp \nabla\eta_\ve} \abs{\frac{u}\ve}^px^A\dx}^{\frac1p}\pt{\int_{\supp \nabla\eta_\ve} \abs{\nabla u}^{p}x^A\dx}^{1-\frac1p}
\end{align*}
but if \(x\in \supp\nabla \eta_\ve\) then \(x_{i_0}\leq 2\ve\) for some \(i_0\in I_2\) and as a consequence
\[
\int_{\supp\nabla\eta_\ve}\abs{\frac{u}{\ve}}^px^A\dx\leq C\sum_{i\in I_2}\int_{\supp\nabla\eta_\ve}\abs{\frac{u}{x_i}}^px^A\dx\leq C\sum_{i\in I_2}\int_{\Omega}\abs{\frac{u}{x_i}}^px^A\dx\leq C\norm{u}_{1,p,A}^p,
\]
by the weighted Sobolev-Hardy inequality \cite[Theorem 1]{Cas2016-2}. Hence
\[
\abs{\int_{\Omega^+} u\abs{\nabla u}^{p-2}\nabla u\cdot\nabla\eta_\ve x^A\dx}\leq C\norm{u}_{1,p,A}\pt{\int_{\supp \nabla\eta_\ve} \abs{\nabla u}^{p}x^A\dx}^{1-\frac1p}.
\]
Similarly we have
\begin{gather*}
\abs{\int_{\Omega^+} v\abs{\nabla u}^{p-2}\nabla u\cdot\nabla\eta_\ve x^A\dx}\leq C\norm{v}_{1,p,A}\pt{\int_{\supp \nabla\eta_\ve} \abs{\nabla u}^{p}x^A\dx}^{1-\frac1p}\\
\abs{\int_{\Omega^+} u\abs{\nabla v}^{p-2}\nabla v\cdot\nabla\eta_\ve x^A\dx}\leq C\norm{u}_{1,p,A}\pt{\int_{\supp \nabla\eta_\ve} \abs{\nabla v}^{p}x^A\dx}^{1-\frac1p}\\
\abs{\int_{\Omega^+} v\abs{\nabla v}^{p-2}\nabla v\cdot\nabla\eta_\ve x^A\dx}\leq C\norm{v}_{1,p,A}\pt{\int_{\supp \nabla\eta_\ve} \abs{\nabla v}^{p}x^A\dx}^{1-\frac1p}.
\end{gather*}
In summary we have obtained
\begin{multline*}
\int_{\Omega^+} \pt{\abs{\nabla u}^{p-2}\nabla u-\abs{\nabla v}^{p-2}\nabla v}\cdot\pt{\nabla u-\nabla v}\eta_\ve x^A\dx\\\leq C\pt{\norm{u}_{1,p,A}+\norm{v}_{1,p,A}}\cdot\spt{\pt{\int_{\supp \nabla\eta_\ve}\abs{\nabla u}^px^A\dx}^{1-\frac1p}+\pt{\int_{\supp \nabla\eta_\ve}\abs{\nabla u}^px^A\dx}^{1-\frac1p}}
\end{multline*}
and as \(\ve\to 0^+\) we deduce that
\[
\int_{\Omega^+} \pt{\abs{\nabla u}^{p-2}\nabla u-\abs{\nabla v}^{p-2}\nabla v}\cdot\pt{\nabla u-\nabla v}x^A\dx\leq 0
\]
so we can proceed as in the proof of \cref{comparison-w} to conclude.
\end{proof}

\providecommand{\bysame}{\leavevmode\hbox to3em{\hrulefill}\thinspace}
\providecommand{\MR}{\relax\ifhmode\unskip\space\fi MR }
\providecommand{\MRhref}[2]{%
	\href{http://www.ams.org/mathscinet-getitem?mr=#1}{#2}
}
\providecommand{\href}[2]{#2}

\end{document}